\documentclass{article}

\usepackage{graphicx}
\usepackage{amssymb}
\usepackage{amsthm}
\usepackage{amsmath}
\usepackage[numbers]{natbib}
\usepackage{enumerate}
\usepackage{mathtools}
\usepackage{authblk}

\usepackage{lineno}
\usepackage[colorinlistoftodos]{todonotes}
\usepackage[colorlinks=true, allcolors=blue]{hyperref}

\newcommand{{\RR}}{\mathbb R}
\newcommand{{\PP}}{\mathbb P}
\newcommand{{\NN}}{\mathbb N}
\newcommand{{\NNN}}{\mathcal N}
\newcommand{{\UUU}}{\mathcal U}
\newcommand{{\EEE}}{\mathcal E}
\newcommand{{\II}}{\mathcal I}
\newcommand{{\DD}}{\mathcal D}

\newcommand{{\PPP}}{\underline P}

\newtheorem{Thm}{Theorem}[section]
\newtheorem{Prop}{Proposition}[section]
\newtheorem{Lemma}{Lemma}[section]
\newtheorem{Cor}{Corollary}[section]
\newtheorem{Rem}{Remark}[section]
\newtheorem{Def}{Definition}[section]
\newtheorem{Ex}{Example}[section]

\begin{document}


\title{Nearly-Linear Uncertainty Measures}





\author[1]{Chiara Corsato\thanks{ccorsato@units.it}}
\author[1]{Renato Pelessoni\thanks{renato.pelessoni@econ.units.it}}
\author[1]{Paolo Vicig\thanks{paolo.vicig@econ.units.it}}
\affil[1]{DEAMS ``B. de Finetti''\\
	University of Trieste\\
	Piazzale Europa~1\\
	I-34127 Trieste\\
	Italy}

\renewcommand\Authands{ and }

\maketitle

\begin{abstract}
Several easy to understand and computationally tractable imprecise probability models, like the Pari-Mutuel model, are derived from a given probability measure $P_0$. In this paper we investigate a family of such models, called Nearly-Linear (NL). They generalise a number of well-known models, while preserving a simple mathematical structure. In fact, they are linear affine transformations of $P_0$ as long as the transformation returns a value in $[0,1]$. We study the properties of NL measures that are (at least) capacities, and show that they can be partitioned into three major subfamilies. We investigate their consistency, which ranges from 2-coherence, the minimal condition satisfied by all, to coherence, 
and the kind of beliefs they can represent. There is a variety of different situations that NL models can incorporate, from generalisations of the Pari-Mutuel model, the $\varepsilon$-contamination model and other models to conflicting attitudes of an agent towards low/high $P_0$-probability events (both prudential and imprudent at the same time), or to symmetry judgments. The consistency properties vary with the beliefs represented, but not strictly: some conflicting and partly irrational moods may be compatible with coherence. In a final part, we compare NL models with their closest, but only partly overlapping, models, neo-additive capacities and probability intervals.

\smallskip
\noindent \textbf{Keywords.}
Nearly-Linear models,
Pari-Mutuel model,
2-coherent imprecise probabilities,
coherent imprecise probabilities,
probability intervals
\end{abstract}



\section*{Acknowledgement}
*NOTICE: This is the authors' version of a work that was accepted for publication in International Journal of Approximate Reasoning. Changes resulting from the publishing process, such as peer review, editing, corrections, structural formatting, and other quality control mechanisms may not be reflected in this document. Changes may have been made to this work since it was submitted for publication. A definitive version was subsequently published in International Journal of Approximate Reasoning, 
vol. 114, November~2019, pages 1–-28
https://doi.org/10.1016/j.ijar.2019.08.001 $\copyright$ Copyright Elsevier

https://doi.org/10.1016/j.ijar.2019.08.001

\vspace{0.3cm}
$\copyright$ 2019. This manuscript version is made available under the CC-BY-NC-ND 4.0 license http://creativecommons.org/licenses/by-nc-nd/4.0/

\begin{center}
	\includegraphics[width=2cm]{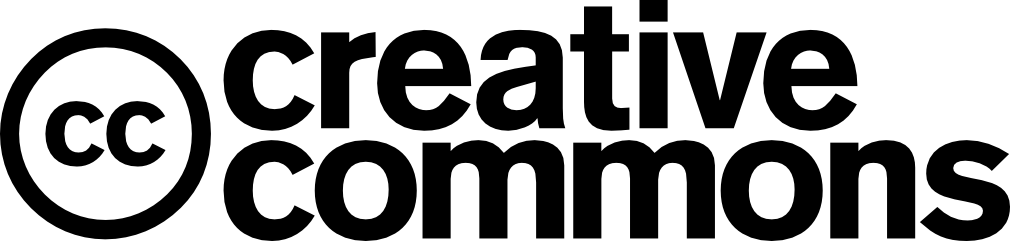}
	\includegraphics[width=2cm]{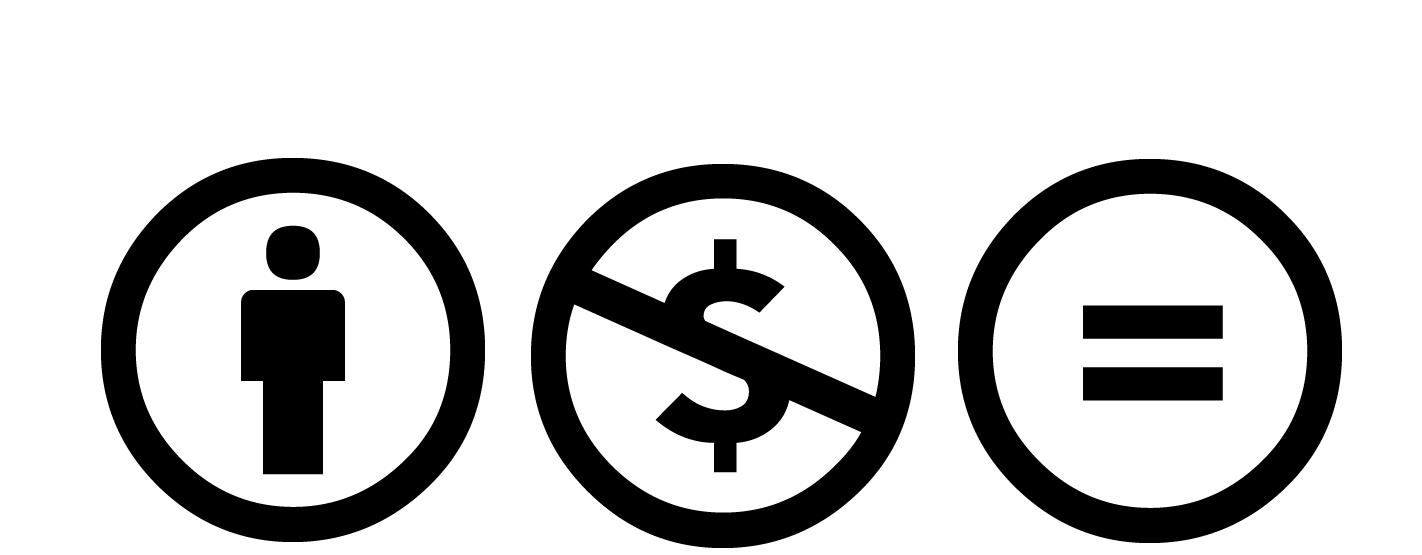}
\end{center}

\section{Introduction}
A great number of different models is nowadays available to represent uncertainty and imprecision in real-world knowledge. They range from very general ones like lower and upper previsions to special cases like probability intervals, $p$-boxes, possibility/necessity measures, and others. When it is reasonable to adopt it, the advantage of a special model is that it is a simplified and easier to understand uncertainty representation. 

In this paper, we introduce and investigate a family of such simplified models, to be called \emph{Nearly-Linear} (NL) models. Formally, NL models are \emph{neighbourhood models} (cf. \cite[Chapter 4.7]{book}), being obtained from a  given probability measure $P_0$. The probability $P_0$ may represent an assessor's first approach evaluation, or a `true' probability that has to be modified for some reason. For instance, $P_0(A)$ is not a bookmaker's realistic selling price\footnote{When the bookmaker sells event $A$, s/he then pays to the buyer 1 if $A$ occurs, 0 otherwise.} for an event $A$: being a fair price, it does not ensure a positive gain in the long run, nor does it incorporate the bookmaker's costs and commissions. Thus, in the Pari-Mutuel model, a well-known neighbourhood model born in the world of horse betting \cite{MMD18,PVZ,W}, $P_0(A)$ is replaced by the upper probability $\overline P_{\rm PMM}(A)=\min\{(1+\delta)P_0(A),1\}$, $\delta>0$, as a more credible selling price for $A$. Note however that $\overline P_{\rm PMM}(A)\downarrow 0$ as $P_0(A)\downarrow 0$, which again might not correspond to reality: a bookie may stand some fixed cost $c>0$, irrespective of how unlikely $A$ is. If $(1+\delta)P_0(A)$ is smaller than $c$, after deducting all costs $c$ the bookmaker's net gain from selling $A$ for $\overline P_{\rm PMM}(A)$ is negative, no matter whether $A$ occurs or not. This problem is solved by the NL model, generalising the Pari-Mutuel model, studied in the later Section \ref{vertical}. 

The idea behind the definition of NL models is rather simple: $\mu$ is a NL uncertainty measure if it is a linear affine transformation of $P_0$, as long as this makes sense, i.e., as long as $\mu\in[0,1]$. Thus, 
\begin{equation}
\label{mu_linear}
    \mu(A)=bP_0(A)+a, \quad \text{if }bP_0(A)+a\in[0,1].
\end{equation}
The main goals of our investigation of the NL models are:
\begin{itemize}
    \item[$(a)$] Determine the general properties of NL models and classify them into exhaustive families of submodels. 
    \item[$(b)$] Detect the consistency properties of each submodel.
    \item[$(c)$] Establish which beliefs they may elicit.
    \item[$(d)$] Compare the NL models with their closest uncertainty models.
\end{itemize}
In detail, after recalling some essential preliminary notions in Section \ref{section_prel}, item $(a)$ is tackled in Section \ref{NL_models}. After observing that requiring $b>0$ in \eqref{mu_linear} guarantees that $\mu$ (defined precisely by the later Equation \eqref{def_mu_NL}) is a capacity, we prove that NL models are closed with respect to conjugacy (Proposition \ref{conjugate}). Thus every NL model corresponds to a couple of conjugate capacities; by establishing a condition for a lower probability $\PPP$ in a NL model to be 2-coherent, Propositions \ref{prop_2coherence} and \ref{prop_sopra2coherence} let us interpret this couple as a lower and an upper probability. From these results, and taking the Pari-Mutuel model as a starting point, NL models can be partitioned into three submodels. The first, the Vertical Barrier Model (VBM) studied in Section \ref{vertical}, is always coherent and extends various well-known models, including the Pari-Mutuel model, the $\varepsilon$-contamination model, and the vacuous imprecise probability. In Section \ref{horizontal} we investigate the Horizontal Barrier Model (HBM). It is always 2-coherent, and may express an assessor's conflicting attitude (prudential towards some events, imprudent towards other ones). Perhaps surprisingly, this may be compatible with coherence too, under additional conditions. We prove in Proposition \ref{characterise_coherence} that the upper probability in this model is coherent if and only if it is subadditive, and in Section \ref{coherent_uplow_HB} we detail the rather restrictive conditions imposed by subadditivity. Some of these models may even be precise probabilities, as discussed in Section \ref{section_precise_HBM}. Section \ref{sub_RRM} studies the third NL model, the Restricted Range Model. It does not generalise the Pari-Mutuel model, may still elicit conflicting moods for an assessor, is always 2-coherent, but coherent only in marginal situations. In Section \ref{HZ} we explore items $(a)$, $(b)$ and $(c)$ for a degenerate NL model, the Hurwicz capacity corresponding to $b=0$ in \eqref{mu_linear}. Item $(d)$ is addressed in Section \ref{sec_compare}: preliminarily, we study in Section \ref{NU} how properties of the sets of events of $\mu$-measure 0 and 1 depend on the kind of consistency of $\mu$. This analysis is useful to compare, in Section \ref{NL_neo}, NL models with neo-additive capacities, introduced in \cite{CEG} in a decision theoretical framework. Although making use of \eqref{mu_linear} too, neo-additive capacities require additional (and, in our opinion, overly restrictive) constraints on the sets of $\mu$-measure 0 and 1. These constraints, among other issues, differentiate them from NL models. A parallel between NL models and probability intervals is also of interest, given that the Pari-Mutuel model is, in a \emph{finite} setting, an instance of probability interval \cite{MMD18}. This comparison is done in Section \ref{sec_intervals}. Its final results (Propositions \ref{prop_characterise_intervals}, \ref{prop_HBM_intervals}) show that a VBM is (the natural extension of) a probability interval in very special instances only, while a HBM is so (always in a finite setting) if it is coherent, which is still a rather special situation. Finally, Section \ref{sec_conclusions} concludes the paper. This work is an extended version, with proofs and additional material, of the contribution \cite{CPV18} presented at the SMPS 2018 conference. Proofs of results are gathered in the Appendix.

\section{Preliminaries}
\label{section_prel}
\subsection{Describing uncertainty}
\label{describe_uncertain}
Quite commonly, uncertainty evaluations are made after fixing a \emph{partition} (also termed space or universe of discourse) $\PP$, i.e., a set of pairwise disjoint events whose logical sum is the sure event $\Omega$. The evaluations concern the set $\mathcal A(\PP)$ of events logically dependent on $\PP$ (the powerset of $\PP$, in set-theoretic language).

In principle, however, one may well think of evaluating the events of an \emph{arbitrary} non-empty set $\mathcal D$, without requiring a priori any structure or constraint on $\mathcal D$. This is customary in de Finetti's approach to Subjective Probability Theory \cite{deF74}, and was to some extent inherited in Imprecise Probability Theory \cite{W}. In this paper we shall quite often put $\mathcal D=\mathcal A(\PP)$, but general definitions and some later development will refer to generic domains $\mathcal D$. 

The two approaches are linked as follows: given a set of events $\mathcal D=\{A_i\}_{i\in I}$, with $I$ an arbitrary (non-empty) index set, we obtain the so-called \emph{partition generated} by $\mathcal D$, $\PP_G=\{\bigwedge_{i\in I} A_i'\}$, where $A_i'$ may be replaced by either $A_i$ or its negation $\neg A_i$, in all possible ways for $i\in I$. The properties of $\PP_G$ are: 
$(a)$ any event in $\mathcal D$ belongs to $\mathcal A(\PP_G)$, and
$(b)$ $\PP_G$ is the coarsest partition with the property $(a)$.

When computing $\PP_G$, several events $\bigwedge_{i\in I}A_i'$ may be impossible ($\bigwedge_{i\in I} A_i'=\emptyset$), depending on the relationships among the events in $\mathcal D$. Clearly, what interests is the partition $\PP_G\setminus\{\emptyset\}$. More generally, we assume that the partitions in the sequel are all made of non-impossible events or \emph{atoms}, i.e., $\omega\in\PP$ implies also $\omega\neq \emptyset$.

\subsection{Measuring uncertainty}
In this paper, we shall encounter several kinds of uncertainty measures, starting with capacities which, given a partition $\PP$, are commonly defined on the set $\mathcal A(\PP)$.

\begin{Def}
\label{def_capac}
An uncertainty measure $\mu:\mathcal A(\PP)\to\RR$ is a normalised capacity, or simply a \emph{capacity}, if $\mu(\emptyset)=0$, $\mu(\Omega)=1$ (normalisation), and for any $A,B\in\mathcal A(\PP)$ such that $A\Rightarrow B$, it holds that $\mu(A)\le \mu(B)$ (monotonicity).
\end{Def}
 Capacities are quite common in several kinds of applications (see, e.g., \cite{G16}). A capacity is sometimes named \emph{fuzzy measure}, although this term is often reserved for capacities with additional continuity or at least semicontinuity properties, which we do not require here.
 
The properties of capacities as uncertainty measures are really minimal. At the other extreme, we find (precise) probabilities. They may be thought of as special imprecise probability assessments where the lower probability $\PPP:\DD\to\RR$ and the upper probability $\overline P:\DD\to\RR$ coincide for any event in $\DD$: $\PPP=\overline P=P$, and satisfy some consistency conditions (see Lemma \ref{zero} later on) \cite{W}. As for upper and lower probabilities, in a first approach, common in the literature (see e.g. \cite[p. 61]{W}), they are formally both maps from $\mathcal D$ into $\RR$. However, what matters to let them represent reasonable assessments is which consistency criterion they satisfy, and here the definitions, within the same criterion, are different for lower or for upper probabilities. See e.g. the later Definition \ref{def_all} $(a)$ and Definition \ref{def_upper_coherence} in the case of coherence.   

It is also customary to assume \emph{conjugacy} of $\PPP$ and $\overline P$, which amounts to the identity
\begin{equation}
\label{conju}
\overline P(\neg A)=1-\PPP(A), \quad \forall A\in\mathcal \DD
\end{equation}
and lets us refer to lower or alternatively upper probabilities only (if $\PPP$ is defined on $\DD$, its conjugate $\overline P$ is defined on $\{\neg A:A\in\DD\}$).

Lower/upper probabilities may satisfy consistency criteria of different strength, and whose properties deviate by various degrees from those of precise probabilities. We group here the ones concerning lower probabilities recalled in the sequel:
\begin{Def}\cite{ PV03, PV16, W}
	\label{def_all}
	Let $\PPP:\DD\to\RR$ be given, and $\NN$ be the set of natural numbers (including 0).
	\begin{itemize}
		\item[$(a)$] $\PPP$ is a \emph{coherent lower probability} on $\DD$ iff, $\forall n\in\NN$, $\forall s_i\ge 0$, $\forall A_i\in\DD$, $i=0,1,\dots,n$, defining
		$$
		\underline G=\sum_{i=1}^n s_i\big(I_{A_i}-\PPP(A_i)\big) - s_0\big(I_{A_0}-\PPP(A_0)\big),
		$$
		it holds that $\max \underline G\ge 0$.
		\item[$(b)$] $\PPP$ is a \emph{convex lower probability} on $\DD$ iff $(a)$ holds with the additional convexity constraint $\sum_{i=1}^ns_i=s_0=1$.
		
		$\PPP$ is centered convex or \emph{C-convex} iff it is convex, $\emptyset\in\DD$ and $\PPP(\emptyset)=0$.
		\item[$(c)$] $\PPP$ \emph{avoids sure loss} on $\DD$ iff $(a)$ holds with $s_0=0$.
		\item[$(d)$] $\PPP$ is \emph{2-coherent} on $\DD$ iff, $\forall s_1\ge 0$, $\forall s_0\in\RR$, $\forall A_0,A_1\in\DD$, defining
		$\underline G_2=s_1\big(I_{A_1}-\PPP(A_1)\big) - s_0\big(I_{A_0}-\PPP(A_0)\big)$, it holds that $\max \underline G_2\ge 0$.
	\end{itemize}
\end{Def}
Although Definition \ref{def_all} is axiomatical (and as such would not require further explanations), it implements the customary interpretation of a lower probability $\PPP(A)$ as an assessor's supremum buying price for $A$ (or for its indicator $I_A$) \cite{book,TdC,W}, while an upper probability $\overline P(A)$ is viewed as an infimum selling price for $A$.\footnote{The assessor is also termed subject, agent, or bettor in the literature. We shall use these terms as synonyms, and with reference to the betting interpretation recalled here.} In fact, a subject paying $\underline P(A)$ for $A$ achieves a(n elementary) random gain $I_A-\PPP(A)$. The consistency notions in Definition~\ref{def_all} require that no finite linear combination of elementary gains, with certain requirements on the coefficients or stakes $s_i$ that vary with the specific notion $(a)$, $(b)$, $(c)$ or $(d)$, is negative, meaning that the overall gain ($\underline G$ or $\underline G_2$) cannot result in a sure loss for the assessor/bettor.

In this interpretation, Definition \ref{def_all} operates a \emph{betting scheme} with variants as for the allowed stakes. We emphasise that the betting scheme is primarily a(n abstract) device for eliciting and graduating an assessor's uncertainty evaluation.

Coherence, a fundamental notion in the approach of \cite{W}, is the strongest among these consistency notions and implies all the other ones. 
The corresponding consistency concepts for upper probabilities may be derived from the definitions above and conjugacy \eqref{conju}. We recall explicitly coherence:
\begin{Def}
\label{def_upper_coherence}
$\overline P:\DD\to\RR$ is a \emph{coherent upper probability} on $\DD$ iff, $\forall n\in\NN$, $\forall s_i\ge 0$, $\forall A_i\in\DD$, $i=0,1,\dots,n$, defining
$$
\overline G=\sum_{i=1}^n s_i\big(\overline P(A_i) - I_{A_i}\big) - s_0\big(\overline P(A_0) - I_{A_0}\big),
$$
it holds that $\max \overline G\ge 0$.
\end{Def}
 Again, Definition \ref{def_upper_coherence} requires that the gain $\overline G$ from any combination of bets at $\overline P$-prices, with non-negative coefficients but for at most one, is not always negative. 

Some necessary conditions for coherence are (whenever the relevant lower/\\upper probabilities are defined) \cite[Section 2.7.4]{W}: 
\begin{align}
\label{subadd}
\overline P(A\vee B) & \le \overline P(A) + \overline P(B) & \text{(\emph{subadditivity})},\\
\label{superadd}
\PPP(A\vee B) & \ge \PPP(A) + \PPP(B), \quad \text{if }A\wedge B=\emptyset & \text{(\emph{superadditivity})},\\
\label{quasi_superadd}
1 + \PPP(A\wedge B) & \ge \PPP(A) + \PPP(B). &
\end{align}

The definitions of coherence for lower and upper probabilities both weaken de Finetti's coherence definition for (precise) probabilities \cite{deF74}, also called dF-coherence:

\begin{Def}
	\label{def_nuova}
	Let $P:\mathcal D\to \RR$ be given. $P$ is a (precise) \emph{probability} \emph{dF-coherent} on $\mathcal D$ iff, $\forall n\in\NN$, $\forall s_i\in\RR$, $\forall A_i\in\mathcal D$, $i=1,\dots,n$, defining
	$$
	G=\sum_{i=1}^n s_i\big(I_{A_i} - P(A_i)\big),
	$$
	it holds that $\max G\ge 0$.
\end{Def}
Note that also dF-coherent probabilities are defined on an \emph{arbitrary} set of events $\mathcal D$. In this paper, we shall use the term (precise) probability to mean dF-coherent probability. 
It is possible to characterise (precise) probabilities as follows \cite[Section 2.8.8]{W}.

\begin{Lemma}
\label{zero}
$P$ is a probability on a set $\mathcal{D}$ closed under negation ($A\in\mathcal{D}\Rightarrow \neg A\in \mathcal{D}$) iff $P$ avoids sure loss (as a lower probability) and $P(A)+P(\neg A)=1$, $\forall A\in \mathcal{D}$. 
\end{Lemma}

\begin{Rem}
\label{rem_zero}
We recall for later applications of this lemma that a coherent lower/upper probability also avoids sure loss, and that if $\PPP=\overline P=P$ and $\PPP,\overline P$ are conjugate, then necessarily $P(A)+P(\neg A)=1$, $\forall A\in\mathcal{D}$. 
\end{Rem}

A lower probability $\PPP$ is \emph{2-monotone} on $\mathcal A(\PP)$ if 
\begin{equation}
\label{2monot}
\PPP(A\vee B)\ge \PPP(A) + \PPP(B) - \PPP(A\wedge B), \quad \forall A,B\in\mathcal A(\PP),
\end{equation}
while an upper probability $\overline P$ is \emph{2-alternating} on $\mathcal A(\PP)$ if the reverse inequality holds:
\begin{equation}
\label{2altern}
\overline P(A\vee B)\le \overline P(A) + \overline P(B) - \overline P(A\wedge B), \quad \forall A,B\in\mathcal A(\PP).
\end{equation}
2-monotonicity is not implied by coherence: if $\PPP$ is coherent it need not satisfy \eqref{2monot} but only the weaker condition \eqref{quasi_superadd}. Conversely, 2-monotonicity of $\PPP$ implies its coherence if further $\PPP(\emptyset)=0$, $\PPP(\Omega)=1$ (otherwise it may not: $\PPP(A)=c\in\RR, \,\forall A\in\mathcal A(\PP)$, is 2-monotone but not coherent) \cite[Corollary 6.16]{TdC}. In the theory of Imprecise Probabilities, the importance of 2-monotonicity is essentially computational \cite[Section 4.3]{book}: it simplifies an important inferential procedure termed natural extension, allowing its computation by means of the Choquet integral \cite[Theorem 6.14]{TdC}. Further, in a finite environment, 2-monotonicity reduces the search of the vertexes of the set of all precise probabilities dominating $\PPP$ (the credal set) to the simpler task of finding the so-called Weber set, see e.g. \cite[Section 5.3]{MMV18}. A behavioural interpretation of 2-monotonicity is available for gambles: it is related to comonotone additivity \cite[Theorem 6.22]{TdC}, an important property in many fields, including risk measurement \cite[Section 2.2.3.6]{DDGK05}.

It can be easily seen that, given $\PPP$ and $\overline P$ that are conjugate, $\overline P$ has the same degree of consistency as $\PPP$: it is coherent (2-coherent, $\dots$) iff $\PPP$ is so, and satisfies \eqref{2altern} iff $\PPP$ satisfies \eqref{2monot}.

Among the coherent models, the Pari-Mutuel model will play a basic role in the sequel. It was originally devised for betting with horse racing, and studied in the framework of imprecise probabilities by \cite{MMD18,  PVZ,W}, among others.
\begin{Def}
\label{def_PMM}
$\underline{P}_{\rm PMM}:\mathcal A(\PP)\to\RR$ is a \emph{Pari-Mutuel lower probability} if \mbox{$\underline{P}_{\rm PMM}(A)=\max\{(1+\delta)P_0(A) - \delta,0\},\, \forall A\in\ \mathcal A(\PP)$,} where $P_0$ is a given probability and $\delta\in \RR^+$. Its conjugate upper probability is $\overline{P}_{\rm PMM}(A)=\min\{(1+\delta)P_0(A),1\}$. $(\underline{P}_{\rm PMM},\overline{P}_{\rm PMM})$ constitute a \emph{Pari-Mutuel Model} (PMM).
\end{Def} 
In the PMM, $\PPP_{\rm PMM}$ is 2-monotone, $\overline P_{\rm PMM}$ is 2-alternating. As for $P_0$, we may think that it is a known or `true' probability, but that does not express a subject's own buying/selling prices. And in fact, in real-world situations a bookie does not sell $A$ for what s/he believes is its \emph{fair} price $P_0(A)$, but for a higher amount $\overline P(A)$. This is indeed necessary, at least to cover the bookie's costs, and further to guarantee a profit in the long run.

Some of the concepts we have recalled can be approached alternatively, by means of Envelope Theorems \cite{PV03,W}. 

\begin{Thm}[Envelope Theorems]
\label{env_thm}
\hspace{1pt}
\begin{itemize}
\item[$(a)$] $\PPP:\DD\to\RR$ is coherent on $\DD$ iff there is a non-empty set  $\mathcal M$ of precise probabilities such that
$$
\PPP(A)=\min\{P(A):P\in\mathcal M\}, \quad \forall A\in \DD.
$$
\item[$(b)$] $\overline P:\DD\to\RR$ is coherent on $\DD$ iff there is a non-empty set  $\mathcal M$ of precise probabilities such that
$$
\overline P(A)=\max\{P(A):P\in\mathcal M\}, \quad \forall A\in \DD.
$$
\item[$(c)$] $\PPP:\DD\to\RR$ is convex on $\DD$ iff there exist a non-empty set  $\mathcal M$ of precise probabilities and a function $\alpha:\mathcal M\to\RR$ such that
$$
\PPP(A)=\min\{P(A)+\alpha(P):P\in\mathcal M\}, \quad \forall A\in \DD.
$$
\end{itemize}
\end{Thm}
Envelope theorems justify a further interpretation of coherent lower probabilities, historically arising from statistical robustness arguments \cite{H81}: if we are uncertain about which is the `true' probability in a set $\mathcal M$, we may prudentially obtain an evaluation $\PPP$ by taking the lower envelope in $\mathcal M$. The probabilities in $\mathcal M$ may also be given by different experts. In the case of convex probabilities, the function $\alpha(P)$ lets us modify/correct the opinion of any single expert.\footnote{A further important motivation for introducing convex probabilities, or previsions when referring to gambles, is their one-to-one correspondence with convex risk measures, see \cite{FS02, PV03} for more information.}

As for 2-coherent lower probabilities, the following characterisation, proven in \cite[Proposition 4]{PVMM16}, will turn to be useful later on. 
\begin{Prop}
\label{prop_character_2coherence}
Let $\PPP:\DD\to\RR^+\cup\{0\}$, with $\DD$ negation-invariant ($A\in\DD\Rightarrow \neg A\in\DD$). Then $\PPP$ is 2-coherent iff it satisfies the following:
\begin{itemize}
\item[$(i)$] $\forall A,B\in\DD,$ $A\Rightarrow B$ implies $\PPP(A)\le \PPP(B)$,
\item[$(ii)$] $\forall A\in\DD$, $\PPP(A) + \PPP(\neg A)\le 1$,
\item[$(iii)$] if $\emptyset \in\DD$, $\PPP(\emptyset)=0$, $\PPP(\Omega)=1$.
\end{itemize}
\end{Prop}

\begin{Rem}
	\label{nuovo}
	If $\mathcal D=\mathcal A(\PP)$, by Proposition \ref{prop_character_2coherence} a $2$-coherent $\PPP$ is a capacity with the additional property $(ii)$. Using $(ii)$ and \eqref{conju}, the conjugate $\overline P$ is such that
	$$
	\overline P(A)=1-\PPP(\neg A)\ge \PPP(A),\quad \forall A\in\mathcal A(\PP).
	$$ 
	In other words, 2-coherence is a minimal consistency concept that guarantees the natural inequality $\overline P\ge \PPP$ for conjugate $\PPP,\overline P$.
	
	Formally, 2-coherence is a special case of $n$-coherence, defined in \cite[Appendix B]{W}. Its properties are studied, also in a desirability approach, in the more general framework of conditional gambles in \cite{PV16}.
\end{Rem}

\section{Nearly-Linear Models}
\label{NL_models}

In this section Nearly-Linear (NL) models are defined, and their basic properties are established.

Let for this $\mu:\mathcal A(\mathbb P)\to \mathbb R$ be either a lower or an upper probability.

\begin{Def}
\label{NL_IP}
$\mu:\mathcal{A}(\PP)\to\RR$ is a \emph{Nearly-Linear (NL)} imprecise probability iff $\mu(\emptyset)=0$, $\mu(\Omega)=1$ and, given a probability $P_0$ on $\mathcal{A}(\PP)$, $a\in\RR$, $b> 0$, $\forall A\in\mathcal{A}(\PP)\setminus\{\emptyset,\Omega\}$, 
\begin{equation} 
\label{def_mu_NL}
\mu(A)\stackrel{\rm def}{=}\min\{\max\{bP_0(A) + a,0\},1\}\big(=\max\{\min\{bP_0(A) + a,1\},0\}\big).
\end{equation}
\end{Def}

Thus, a NL imprecise probability is a linear affine transformation with barriers of the probability $P_0$. The barriers prevent the transformation from taking values outside the interval $[0,1]$.

\begin{Lemma}
\label{lemma_capacity}
A NL $\mu$ is a capacity.
\end{Lemma}

\begin{Rem}
\label{b_neg}
The case $b<0$ is ruled out from Definition \ref{NL_IP} to ensure monotonicity of $\mu$ without further, overly restrictive requirements. In fact, if $b<0$ and $A\Rightarrow B$, it is $\mu(A)>\mu(B)$ whenever $\mu(A)=bP_0(A) + a>0$ and either $0<P_0(A)<P_0(B)$ or $P_0(A)=0<P_0(B),\, b>-\frac{a}{P_0(B)}$, but it is also possible that $\mu(A)=1>\mu(B)$, for appropriate values of $P_0,a,b$.

By contrast, $\mu$ is a capacity when $b=0$ if further $0\le a\le 1$. 
We did not include this special situation in Definition \ref{NL_IP}, but shall discuss it in Section \ref{HZ}.
\end{Rem}
If $\mu$ is given by Definition \ref{NL_IP}, we shall say shortly that $\mu$ is NL$(a,b)$.

An interesting feature of the family of NL models is that it is \emph{closed with respect to conjugacy}: if $\mu$ is NL$(a,b)$, also its conjugate $\mu^c(A)=1-\mu(\neg A)$, $\forall A\in\mathcal{A}(\PP)$, is NL$(c,b)$:

\begin{Prop}
\label{conjugate}
If $\mu$ is NL$(a,b)$, then $\mu^c$ is NL$(c,b)$, with 
\begin{equation}
\label{c}
c=1-(a+b).
\end{equation} 
\end{Prop}

By Proposition \ref{conjugate}, a NL model is made of a couple of conjugate capacities. As we shall see as a follow-up of the next propositions (see especially the comment after Proposition \ref{prop_sopra2coherence} and the conclusions in item $(i)$, Section \ref{NL_sub}), they can be interpreted as a lower and its conjugate upper probability, and $\PPP\le\overline P$. We start by preliminarily fixing the notation in the following

\begin{Def}
\label{NL_mod}
A \emph{Nearly-Linear Model} is a couple $(\PPP,\overline P)$ of conjugate lower and upper probabilities on $\mathcal A(\PP)$ where $\PPP$ (hence also $\overline P$, by Proposition \ref{conjugate}) is a NL imprecise probability, usually denoted by NL$(a,b)$, while $\overline P$ is NL$(c,b)$, with $c$ given by \eqref{c}.
\end{Def}

\begin{Ex}
\label{ex_PMM}
In the PMM, $\underline P_{\rm PMM}$ is NL$(-\delta,1+\delta)$, $\overline P_{\rm PMM}$ is NL$(0,1+\delta)$. Here $a=-\delta, \, b=1+\delta,$ and $c=0$, corresponding to $a+b=1$.
\end{Ex}

Given a couple $(\mu,\mu^c)$, Definition \ref{NL_mod} is uninformative as to which between $\mu$ and $\mu^c$ should be regarded as a lower probability. Clearly, when, say, $\mu$ includes a known model $\mu^*$ as a special case, such as for instance the lower probability of a PMM, $\mu$ will be a lower or upper probability if $\mu^*$ is so. More generally, the \emph{maximum consistency principle} may be applied: $\mu$ is a lower probability if it determines a model with a higher degree of consistency than interpreting $\mu$ as an upper probability.

Applying the maximum consistency principle makes sense because the consistency properties of a NL $\mu$ may vary, depending on the choice of the parameters $a,b$. This appears already in the following Proposition \ref{prop_2coherence}. Prior to it, we define the sets $\NNN_{\mu},\UUU_{\mu},\EEE_{\mu}$, of, respectively, \emph{null, universal, essential} events according to $\mu$ (following the terminology in \cite{CEG}).

\begin{Def}
Given an uncertainty measure $\mu$, define:
\begin{align*}
\NNN_{\mu} & = \{A\in\mathcal A(\PP): \mu(A)=0\},\\
\UUU_{\mu} & = \{A\in\mathcal A(\PP): \mu(A)=1\},\\
\EEE_{\mu} & = \mathcal A(\PP)\setminus(\NNN_{\mu}\cup \UUU_{\mu}).
\end{align*}
\end{Def}

When $\mu$ is NL$(a,b)$, we can easily describe these sets in terms of $a,b$, using \eqref{def_mu_NL}:
\begin{align}
\label{N}
\NNN_{\mu} & = \big\{A\in\mathcal A(\PP): P_0(A)\le -\tfrac{a}{b}\big\}\cup \{\emptyset\},\\
\label{U}
\UUU_{\mu} & = \big\{A\in\mathcal A(\PP): P_0(A)\ge \tfrac{1-a}{b}\big\}\cup \{\Omega\},\\
\label{Eext}
\EEE_{\mu} & = \big\{A\in\mathcal A(\PP)\setminus\{\emptyset,\Omega\}: -\tfrac{a}{b}<P_0(A)< \tfrac{1-a}{b}\big\}.
\end{align}

Since a NL model typically gives extreme evaluations to a number of events whose probability $P_0$ is strictly between 0 and 1, determining $\NNN_{\mu}, \UUU_{\mu}, \EEE_{\mu}$ informs us more precisely on this aspect. We shall discuss the structure of $\NNN_{\mu}$ and $\UUU_{\mu}$ in a more general context, in Section \ref{NU}.

\begin{Prop}
\label{prop_2coherence}
Let $\PPP$ be a NL$(a,b)$ lower probability on $\mathcal A(\PP)$. Then $\PPP$ is 2-coherent if 
\begin{equation}
\label{2coherence}
b+2a\le 1.
\end{equation}
\end{Prop}

The results in the next lemma are related to Proposition \ref{prop_2coherence}:

\begin{Lemma}
\label{lemma_link_sottosopraNL}
Let $(\PPP,\overline P)$ be a NL model. Then 
\begin{itemize}
\item[$(a)$] $b+2a\le 1$ iff $b+2c\ge 1$.
\item[$(b)$] $\forall A\in \EEE_{\PPP}\cap \EEE_{\overline P}$, $\overline P(A)- \PPP(A)=1-(b+2a)$.
\end{itemize}
\end{Lemma}

\begin{proof}
	Immediate, using \eqref{c} for $(a)$, and \eqref{conju}, \eqref{def_mu_NL} for $(b)$.
\end{proof}

Recalling that the conjugate $\mu^c$ of $\mu$ has the same consistency properties of $\mu$, by Lemma \ref{lemma_link_sottosopraNL} $(a)$ we get the upper probability version of Proposition~\ref{prop_2coherence}:

\begin{Prop}
\label{prop_sopra2coherence}
Let $\overline P$ be a NL$(c,b)$ upper probability on $\mathcal A(\PP)$. Then $\overline P$ is 2-coherent if $b+2c\ge 1$. 
\end{Prop}

\emph{Comment.} As for Lemma \ref{lemma_link_sottosopraNL} $(b)$, it tells us that the \emph{imprecision} $\overline P(A) -\PPP(A)$ of a NL model is constant on $\EEE_{\PPP}\cap \EEE_{\overline P}$. Clearly, when we speak of imprecision in these terms it is understood that $\overline P\ge \PPP$. In Lemma \ref{lemma_link_sottosopraNL} $(b)$, this is true for $A\in \EEE_{\PPP}\cap \EEE_{\overline P}$ iff \eqref{2coherence} holds. More generally, if \eqref{2coherence} holds, then $\overline P(A)\ge\PPP(A)$ for any $A\in\mathcal A(\PP)$ is implied by 2-coherence of $\PPP$, see Remark \ref{nuovo}. By contrast, $\overline P\ge \PPP$ is not guaranteed  if $(\PPP,\overline P)$ is a generic couple of conjugate measures. However, we shall see later on in the paper that all NL measures can be given an interpretation as either lower or upper probabilities that are (at least) 2-coherent. 

The equality $b+2a=1$ is a limiting situation in Proposition \ref{prop_2coherence}. The next proposition and comments provide more insight for this case.

\begin{Prop}
\label{equiv_sotto=sopra}
Let $(\PPP,\overline P)$ be a NL model (Definition \ref{NL_mod}). 
\begin{itemize}
\item[$(a)$] If $b+2a=1$, then 
$\PPP=\overline P$. 
\item[$(b)$] If $b+2a\neq 1$ and $\PPP=\overline P=P$, then $P$ is 0-1 valued.
\end{itemize}
\end{Prop}




From a first glance at Proposition \ref{equiv_sotto=sopra} $(a)$ one might be tempted to infer that condition $b+2a=1$ is enough to obtain NL models where $\PPP=\overline P=P$ is a precise probability. Yet, this is very often not the case. Just think for this that $\PPP$ may distort $P_0$ so that there exist $k$ pairwise disjoint events $A_1,\dots,A_k$ such that $\PPP(A_i)=0$, $i=1,\dots, k$, while $\PPP(
\bigvee_{i=1}^k A_i)>0$, which makes $\PPP$ non-additive. For instance, take (with $0<-a<b$) $P_0(A_i)\in\,]-\tfrac{a}{kb},-\tfrac{a}{b}[$, $i=1,\dots,k$. Then \eqref{N} ensures that $\PPP(A_i)=0$, $\forall i$, while $\PPP(\bigvee_{i=1}^k A_i)>0$ since $P_0(\bigvee_{i=1}^k A_i)>-\tfrac{a}{kb}k=-\tfrac{a}{b}$. For a result pointing out the constraints to obtain a probability $P$, see the later Proposition \ref{prop_N}.

On the other hand, we may obtain a probability $P$ also when $b+2a\neq 1$, but then $P$ is necessarily 0-1 valued by Proposition \ref{equiv_sotto=sopra} $(b)$. Thus $P$ is concentrated on a single atom of $\PP$, when $\PP$ is finite, while it is possible that $P(\omega)=0$, $\forall \omega\in\PP$, when $\PP$ is infinite. 

We remark that $\PPP=\overline P$ may also be 0-1 valued but not a precise probability, while $\PPP$ and $\overline P$ may or may not be coherent as imprecise probabilities (cf. the later Example \ref{01_not_coherent}).

In general, the role of precise probabilities within NL models is essentially marginal, as we shall also see when studying the various NL submodels.



\subsection{Nearly-Linear submodels}
\label{NL_sub}
In the next sections we shall investigate the submodels forming the family of NL models. Our previous results are useful for determining such submodels. In fact, consider a generic NL measure $\mu(p_{ac},b)$, where the parameter $p_{ac}$ can be either $a$ (when we interpret $\mu$ as a lower probability) or $c$ (when $\mu$ is an upper probability). Then:
\begin{itemize}
\item[$(i)$] Propositions \ref{prop_2coherence}, \ref{prop_sopra2coherence} and Lemma \ref{lemma_link_sottosopraNL} $(a)$ suggest that $\mu$ should be regarded as a lower or upper probability according to whether, respectively,  $b+2p_{ac}<1$ or $b+2p_{ac}>1$ ($\mu=\overline P=\PPP$ when $b+2p_{ac}=1$, by Proposition \ref{equiv_sotto=sopra}). This ensures 2-coherence of $\mu$, $\mu^c$, and hence the very desirable property $\overline P\ge \PPP$.
\item[$(ii)$] Example \ref{ex_PMM} shows that the PMM is a relevant special case of NL model, with $a+b=1$. This suggests modifying this equality to either $a+b\le 1$ or $a+b\ge 1$ (while keeping throughout $b>0$ by Remark \ref{b_neg}), in order to classify NL models that generalise the PMM.
\item[$(iii)$] In the PMM, $a<0$ and $c=0$. Thus we have to consider the relaxations $a\ge 0$, $c\neq 0$. Note that, by \eqref{c}, $c>0, c=0, c<0$ iff, respectively, $a+b<1,a+b=1,a+b>1$. Therefore, $(ii)$ and $(iii)$ require checking whether $a+b<1,a+b=1,a+b>1$, and the sign of $a$. 
\end{itemize}
Thus, given $\mu(p_{ac},b)$, we may first fix $p_{ac}=a$ if $b+2p_{ac}<1$, $p_{ac}=c$ if $b+2p_{ac}>1$, then check whether $p_{ac}>0,p_{ac}<0$ and $b+p_{ac}>1,b+p_{ac}<1$. If $p_{ac}=a$, then conditions $p_{ac}>0$ and $p_{ac} +b >1$ become respectively $a>0$, $a+b>1$; if $p_{ac}=c$, they reduce to, respectively, $a+b<1$, $a<0$. Given that we are exploring three alternatives, following $(i)$, $(ii)$, $(iii)$ (we omit for the moment the limit situations such as $b+2a=1$ - they are treated within the submodel study in the next sections), there are at most 8 distinct cases. Actually, they give rise to three NL (sub)models. We list them in Table \ref{NLmodels}.

\begin{table}[htbp!]
\begin{center}
\begin{tabular}{c|c|c|c|c}
Case & Parameter constraints & $\mu$ & $p_{ac}$ & Model\\
&&&&\\[-1em] \hline &&&&\\[-1em]
1 & $b+ 2p_{ac}<1, \, p_{ac}<0,\, p_{ac} + b<1$ & $\PPP$ & $a$ & VBM\\
&&&&\\[-1em] \hline &&&&\\[-1em] 
2 & $b+ 2p_{ac}>1, \, p_{ac}>0,\, p_{ac} + b>1$ & $\overline P$ & $c$ & VBM\\
&&&&\\[-1em] \hline &&&&\\[-1em] 
3 & $b+ 2p_{ac}<1, \, p_{ac}<0,\, p_{ac} + b>1$ & $\PPP$ & $a$ & HBM\\
&&&&\\[-1em] \hline &&&&\\[-1em] 
4,5 & $b+ 2p_{ac}>1, \, p_{ac}<0,\, p_{ac} + b 
\neq 1$ & $\overline P$ & $c$ & HBM\\
&&&&\\[-1em] \hline &&&&\\[-1em] 
6 & $b+ 2p_{ac}<1, \, p_{ac}>0,\, p_{ac} + b<1$ & $\PPP$ & $a$ & RRM\\
&&&&\\[-1em] \hline &&&&\\[-1em] 
7 & $b+ 2p_{ac}>1, \, p_{ac}>0,\, p_{ac} + b<1$ & $\overline P$ & $c$ & RRM\\
&&&&\\[-1em] \hline &&&&\\[-1em] 
8 & $b+ 2p_{ac}<1, \, p_{ac}>0,\, p_{ac} + b>1$ & $-$ & $-$ & Impossible \\
\end{tabular}

\end{center}
\caption{
The Nearly-Linear (sub)models. For each model, the parameter constraints for its lower and upper probability are described in two consecutive lines (limit situations are omitted here). See Section \ref{vertical} for the Vertical Barrier Model (VBM), Section \ref{horizontal} for the Horizontal Barrier Model (HBM), Section \ref{sub_RRM} for the Restricted Range Model (RRM).}
\label{NLmodels}
\end{table}
\subsection{Evaluating Uncertainty with NL Models}

A NL model is a neighbourhood model, meaning that it obtains an uncertainty evaluation from a given probability $P_0$ by modifying it. Often, $P_0$ may be a reliable uncertainty assessment, possibly given by some expert or arising from symmetry judgements. Thus, why should $P_0$ be altered to obtain $\overline P$, $\PPP$ from it? The reason is that, in a betting scheme, $P_0$ is unfit as a selling price, because, being a fair price, it does not ensure an expected positive gain to the seller \cite{deF74}. A symmetric argument may be brought forth for introducing $\PPP$. In both cases, the betting scheme aims to model real situations, from the world of betting but not only: for instance, the role of `bookie' might be that of an insurer, a broker, and so on. We have already recalled in the Introduction a motivation of this kind for using the PMM. In the next sections, we shall see that NL models offer various ways of modifying $P_0$, that may correspond to an assessor's different beliefs and attitudes, in particular towards events of extreme or nearly $P_0$-probability, and that some of them may be conflicting. 

\section{The Vertical Barrier Model}
\label{vertical}
In this section we introduce a first family of NL models. We require that $\mu$ is a NL$(a,b)$ measure such that 
\begin{equation}
\label{ab_VB}
0\le a+b\le 1,\quad a\le 0.
\end{equation}
It is easy to see that conditions \eqref{ab_VB} may be obtained from Case 1 in Table \ref{NLmodels}, relaxing the strict inequalities there. 

Recalling $(ii)$ in Section \ref{NL_sub}, conditions \eqref{ab_VB} provide a relaxation of the PMM parameters. In fact, when $a+b=1$ and $a=-\delta<0$, hence $b=1+\delta>0$, $\mu$ is the lower probability of a PMM (Example \ref{ex_PMM}).\footnote{When $a=0,\,b=1$, $\mu$ is the probability $P_0$. We shall hereafter neglect this subcase.}

Further, putting $a=0,\, b=\varepsilon<1$, we obtain the lower probability of the \emph{$\varepsilon$-contamination model} (also termed linear-vacuous mixture in \cite{W}):
$$
\underline P(A)=\varepsilon P_0(A), \quad \forall A\in\mathcal{A}(\PP)\setminus\{\Omega\}, \quad \underline P(\Omega)=1.
$$ 
Lastly, when $a+b=0$, from \eqref{N} it is $\NNN_{\mu}=\mathcal A(\PP)\setminus\{\Omega\}$, i.e., $\mu$ is the \emph{vacuous lower probability} $\PPP_V$, equal to 0, $\forall A\in\mathcal A(\PP)\setminus\{\Omega\}$. It is easily seen that allowing $a+b$ to be negative would make us obtain again only $\PPP_V$, hence we canceled this choice for $a,b$. 

Thus, it is clear that $\mu$ is a lower probability. 
It is also immediate to recognise that $\PPP$ (hence $\overline P$) is 2-coherent, by Proposition \ref{prop_2coherence}: Equation \eqref{ab_VB} implies that $b+2a\le 1$.

Since $bP_0(A)+a\le a+b\le 1$, for all $A\in\mathcal A(\PP)$, Equation \eqref{def_mu_NL} simplifies to 
$$
\mu(A)=\max\{bP_0(A) + a,0\}, \quad \forall A\in \mathcal A(\PP)\setminus\{\Omega\}
$$
(note that $\mu(\emptyset)$ is also computed with this formula).
The conjugate upper probability is easily obtained using \eqref{c}. Summing up, we define 
\begin{Def}
\label{VBM}
A \emph{Vertical Barrier Model (VBM)} is a NL model where $\underline P$ and its conjugate $\overline P$ are given by:
\begin{align}
\label{lower_VBM}
\PPP(A) & =\max\{bP_0(A)+a,0\}, & \forall A & \in\mathcal A(\PP)\setminus\{\Omega\}, &  \PPP(\Omega)=1,\\
\label{upper_VBM}
\overline P(A) & =\min\{bP_0(A)+c,1\}, & \forall A & \in\mathcal A(\PP)\setminus\{\emptyset\}, & \overline P(\emptyset)=0,
\end{align}
with $a,b$ satisfying \eqref{ab_VB} and $c$ given by \eqref{c}.   
\end{Def}
A VBM satisfies stronger consistency properties than 2-coherence:

\begin{Prop}
\label{VBM_coherent}
In a VBM,  $\underline P$ and $\overline P$ are coherent. Further, $\underline P$ is 2-monotone, $\overline P$ is 2-alternating. 
\end{Prop}

In the next proposition, we explore how $\PPP$ and $\overline P$ are related to $P_0$ in a VBM. It can be easily proven, see also \cite[Section 4]{CPV18}.

\begin{Prop}
\label{properties_VB}
Let $(\PPP,\overline P)$ be a VBM. Then, concerning $\overline P$:
\begin{itemize}
\item[$i)$] $\overline P(A)\ge P_0(A),\forall A$; 
\item[$ii)$] $\overline P(A)\downarrow c \ge 0$ as $P_0(A)\downarrow 0$; 
\item[$iii)$] $\overline P(A)=1$ iff $P_0(A)\ge \frac{1-c}{b}=\frac{b+a}{b}$. 
\end{itemize}
Correspondingly, for $\PPP$:
\begin{itemize}
\item[$i')$] $\underline P(A)\le P_0(A),\, \forall A$;
\item[$ii')$] $\underline P(A)\uparrow a+b\le 1$ as $P_0(A)\uparrow 1$;
\item[$iii')$] $\underline P(A)=0$ iff $P_0(A)\le -\frac{a}{b}$.
\end{itemize}
\end{Prop}
Firstly let us discuss $\overline P$, comparing it with its special case $c=0$, i.e., $a+b=1$ (and $b>1$), which specialises $\overline P(A)$ into $\overline P_{\rm PMM}(A)=\min\{bP_0(A),1\}$. 
In the betting interpretation, a subject assessing either $\overline P$ or $\overline P_{\rm PMM}$ is essentially unwilling to sell events whose reference or `true' probability $P_0$ is too high, by $iii)$, and in any case her/his selling price is not less than the `fair' price $P_0$, by $i)$. Unlike $\overline P_{\rm PMM}$, $\overline P$ adds a further barrier regarding low probability events: by $ii)$, if $c>0$ the $\overline P$-agent is not willing to sell (too) low probability events for less than $c$. We may deduce that, \emph{ceteris paribus}, the $\overline P$-agent is, loosely speaking, greedier than the $\overline P_{\rm PMM}$-agent. This can be easily justified in real-world situations: if the agent is, for instance, a bookmaker or an insurer, $c>0$ may take account of the agent's fixed costs in managing any bet/contract. 

A representation in the $(P_0, \overline P)$ plane is helpful to visualise the above facts. In Figure \ref{fig_VBM}, 1), a VBM $\overline P(c,b)$, with $c<1$, is drawn with a continuous bold line. It is compared with a PMM upper probability (dashed bold line) $\overline P_{{\rm{PMM}}}(A)=\min\{b'P_0(A),1\}$. Choosing $b'=\frac{b}{1-c}$ emphasises at best the core difference between $\overline P$ and $\overline P_{{\rm {PMM}}}$: while the horizontal barrier on the line $\overline P=1$ is in this case the same for both models, the VBM originates the vertical barrier of length $c$ on the $\overline P$ axis (dotted segment), not existing with any $\overline P_{\rm PMM}$. 

In general, while $c$ measures the agent's advantage at $P_0=0$, $b$ determines how it varies with $P_0$ growing. In fact, the advantage is unchanged, decreasing or increasing according to whether it is, respectively, $b=1, \, b<1, \, b>1$. 

Turning to the interpretation of $\PPP$ in the VBM, now the $\PPP$-agent acts as a buyer, but by $ii')$ does not want to pay more than $a+b$ for any event, even those whose probability $P_0$ is very high. If $a+b<1$, this ensures that the agent's gain from the transaction, $\underline G=I_A-\PPP(A)$, has a positive maximum (achieved when $A$ occurs): $\max \underline G=1-\PPP(A)\ge 1-(a+b)>0$.  By contrast, $\max \underline G\to 0$ as $P_0(A)\to 1$ if $a+b=1$, as in the PMM. Thus, $\underline P$ in the typical VBM (i.e., such that $a+b<1$) introduces an additional vertical barrier with respect to $\underline P_{\rm PMM}$: the dotted segment in the $P_0=1$ line of Figure \ref{fig_VBM}, $2)$. The parameters $a,b$ jointly influence the barrier width $1-(a+b)$. The ratio $-\frac{a}{b}$ is the supremum probability $P_0(A)$ such that the agent is not willing to buy $A$ unless for free. If $-\frac{a}{b}\ge \frac{1}{2}$, the agent's attitude is \emph{over-prudential}: an event $A$ is bought only if its probability $P_0$ is higher than that of its negation.

\begin{figure}
\centering
\includegraphics[width=\textwidth]{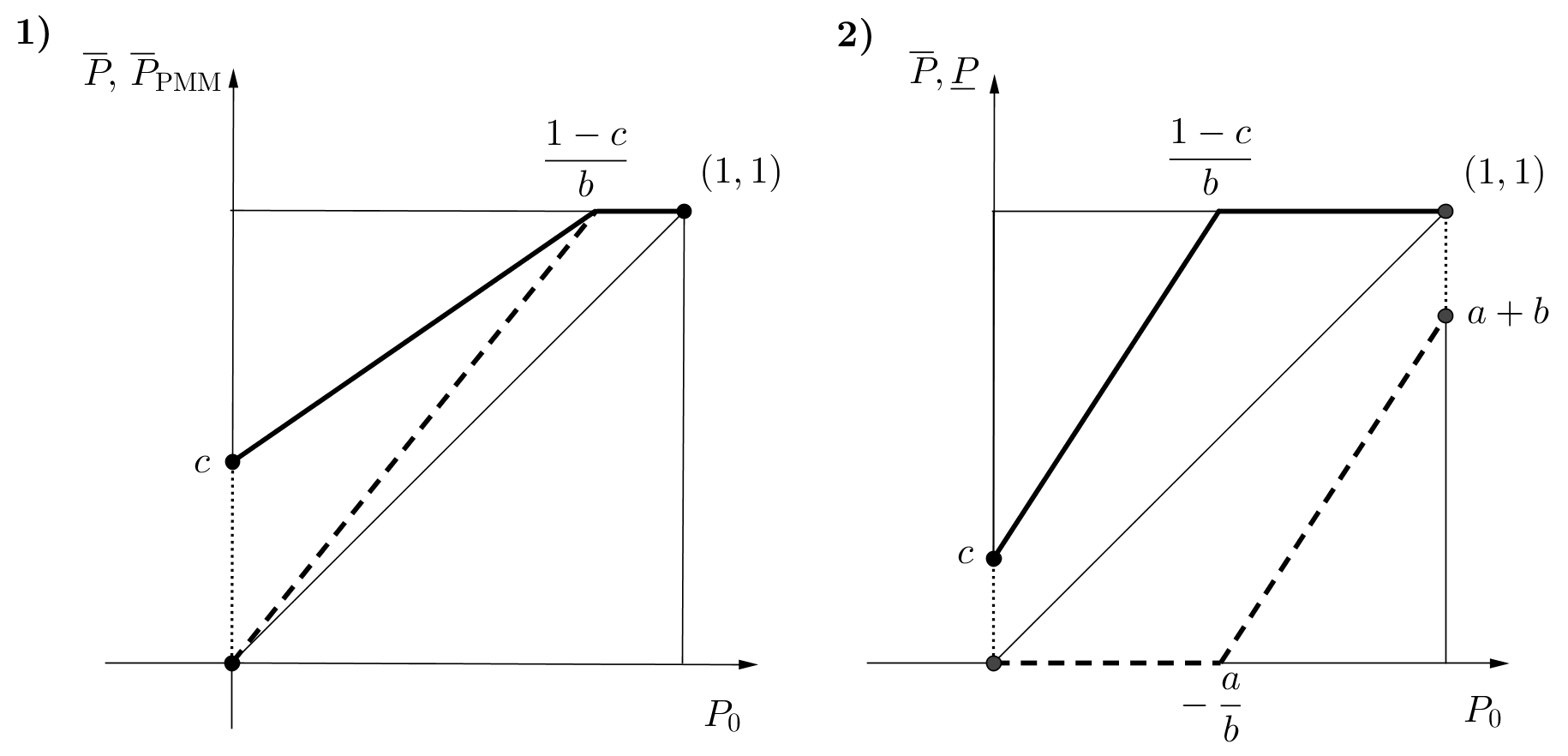}
\caption{$1)$ Plots of $\overline P(c,b)$ (continuous bold line) and $\overline P_{\rm PMM}(0,b')$ (dashed bold line), with $b'=\frac{b}{1-c}$, against $P_0$ ($\overline P,\overline P_{\rm PMM}$ overlap at the line $\overline P=\overline P_{\rm PMM}=1$). $2)$ Plots of $\overline P$ (continuous bold line) and its conjugate $\PPP$ (dashed bold line) against $P_0$.}
\label{fig_VBM}
\end{figure}

\begin{Rem}
The representations in the $(P_0,\PPP)$ or $(P_0,\overline P)$ plane do not imply that $\PPP$ or $\overline P$ is always a function of $P_0$. For instance, take the VBM with $a+b<1$: when $P_0=1$, $\PPP$ may take the value $a+b$ but also (at $\Omega$) the value 1. 
\end{Rem}

As for the existence of precise probabilities within the VBM, from Proposition \ref{properties_VB} $i),i')$, any such probability $P=\PPP=\overline P$ coincides with $P_0$, which is then the only precise probability within this model.

\section{The Horizontal Barrier Model}
\label{horizontal}
To introduce a second family of NL models, let now $\mu$ be a NL$(a,b)$ measure, with the conditions
\begin{equation}
\label{cond_HB}
a+b\ge 1, \quad b+2a\le 1.
\end{equation}
While the VBM relaxes the PMM equality $a+b=1$ to $a+b\le 1$, now the opposite generalisation $a+b\ge 1$ is made. Further, the VBM condition $a\le 0$ in \eqref{ab_VB} ensures that $b+2a\le 1$ there, while here this is explicitly required in \eqref{cond_HB}. Conditions \eqref{cond_HB} are also straightforwardly obtained from $\mu(p_{ac},b)$ satisfying the (relaxed) parameter constraints of Case 3 in Table \ref{NLmodels}.

Again, $a+b=1$ implies $a\le 0$ and thus we reobtain the PMM. For the sequel, we modify \eqref{cond_HB} to
\begin{equation}
\label{strict_cond_HB}
a+b>1,\quad b+2a\le 1,
\end{equation}
to focus on the models in this family that are not PMMs.

Note that conditions \eqref{strict_cond_HB} imply easily that
\begin{equation}
\label{sign_ab}
a<0, \quad b>1.
\end{equation}

\begin{Prop}
\label{HB_2coherent}
Let $\mu(a,b):\mathcal A(\PP)\to \RR$ be a NL measure satisfying \eqref{strict_cond_HB}. Then $\mu$ is a 2-coherent lower probability, and its conjugate $\mu^c(c,b)$ is a 2-coherent upper probability. 
\end{Prop}

\begin{proof}
	Follows directly from Proposition \ref{prop_2coherence}, Lemma \ref{lemma_link_sottosopraNL} $(a)$ and Proposition \ref{prop_sopra2coherence}.
\end{proof}
From Proposition \ref{HB_2coherent}, $\mu \, (\mu^c)$ is conveniently viewed as a lower (an upper) probability. We define then:
\begin{Def}
\label{HBM_mod}
A \emph{Horizontal Barrier Model (HBM)} is a NL model where $a,b$ are as in \eqref{strict_cond_HB} (as in \eqref{cond_HB} if we wanted to include PMMs), $c=1-(a+b)<0$ and, $\forall A\in\mathcal{A}(\PP)\setminus\{\emptyset,\Omega\}$, 
\begin{align}
\label{lower_HBM}
\PPP(A) & =\min\{\max\{bP_0(A) + a,0\},1\},\\
\label{upper_HBM}
\overline P(A) & =\max\{\min\{bP_0(A) + c,1\},0\}. 
\end{align} 
\end{Def}

\subsection{Beliefs elicited by a HBM}
To clarify what sort of beliefs are conveyed by a HBM, let us first state some properties of these models. They are easily established.

\begin{Prop}
\label{link_HB_P0}
Let $(\PPP,\overline P)$ be a HBM. Then, concerning $\PPP$
\begin{itemize}
\item[$j)$] 
$\PPP(A)> P_0(A)$ iff $1>P_0(A)> -\frac{a}{b-1};$
\item[$jj)$] $\PPP(A)=0$ iff $P_0(A)\le -\frac{a}{b}; \quad \NNN_{P_0}\subset \NNN_{\PPP}$;
\item[$jjj)$] $\PPP(A)=1$ iff $P_0(A)\ge\frac{1-a}{b}; \quad \UUU_{P_0}\subset \UUU_{\PPP}$.
\end{itemize}
As for $\overline P$,
\begin{itemize}
\item[$j')$] 
$\overline P(A)<P_0(A)$ iff $0<P_0(A)< -\frac{c}{b-1};$
\item[$jj')$] $\overline P(A)=0$ iff $P_0(A)\le -\frac{c}{b}; \quad \NNN_{P_0}\subset \NNN_{\overline P}$;
\item[$jjj')$] $\overline P(A)=1$ iff $P_0(A)\ge\frac{1-c}{b}; \quad \UUU_{P_0}\subset \UUU_{\overline P}$.
\end{itemize}
\end{Prop}

It is also easy to see (using \eqref{strict_cond_HB}, \eqref{sign_ab}) that the conditions in Proposition \ref{link_HB_P0} are not vacuous, i.e., may be satisfied by some event \cite{CPV18}. 

Conditions $j),jj),jjj)$ point out an interesting feature of $\PPP$ in the HBM: the beliefs it represents may be conflicting and, partly, irrational. In fact, assuming again that $P_0$ is the `true' probability for the events in $\mathcal A(\PP)$, by $j)$ the assessor is willing to buy some events for less, others for more than their probability $P_0$. In the extreme situations, by $jj)$ and $jjj)$, the assessor would not buy events whose probability is too low, whilst would certainly buy a high probability event $A$ at the price of 1, gaining from the transaction at most 0 (if $A$ occurs). Thus the assessor underestimates the potential losses of a transaction regarding high probability events, but overestimates them  with low probability events. Note also that $\PPP$ broadens, with respect to $P_0$, both the set of null events ($\NNN_{P_0}\subset \NNN_{\PPP}$) and the set of universal events ($\UUU_{P_0}\subset \UUU_{\PPP}$).

It is then natural to wonder whether the model limits somehow its non-prudential side, or which of the conflicting moods prevails. In some sense, the prudential one.

In fact, while the assessor's behaviour can be more prudential than $P_0$ suggests on both $A$ and its negation $\neg A$, s/he cannot offer a higher price than $P_0$ for both. If, for instance, s/he does so for $A$, so that $\PPP(A)>P_0(A)$, then necessarily $\PPP(\neg A)<P_0(\neg A)$: if not, $\PPP(A) + \PPP(\neg A)>P_0(A) + P_0(\neg A)=1$, a contradiction by Proposition \ref{prop_character_2coherence} $(ii)$, since $\PPP$ is 2-coherent.

Secondly, note that by $jj)$ and $jjj)$ the HBM sets up two horizontal barriers in  the $(P_0,\underline P)$ plane (cf.  Figure \ref{fig_HBM}, $1)$). The lower (prudential) barrier is a segment with measure $-\frac{a}{b}$, the upper barrier (in the imprudent area) a  segment measuring $1-\frac{1-a}{b}$. Outside the limit situation $b+2a=1$, the upper barrier is narrower: $-\frac{a}{b}>1-\frac{1-a}{b}$ iff $b+2a<1$, true by \eqref{strict_cond_HB}. Similarly, the boundary probability $P_0$ between the opposite attitudes is set at $-\frac{a}{b-1}$, larger or at worst equal to $\frac{1}{2}$, by  \eqref{strict_cond_HB}. Although these  properties have a partly qualitative flavour, being uninformative about how many events $A$ are such that $\PPP(A)>P_0(A)$ and about their `measure of imprudence' $\PPP(A)-P_0(A),$ they nevertheless represent a restriction to the imprudent attitude.

To guarantee that there may be events $A$ such that $\PPP(A)>P_0(A)$, the condition $a+b>1$ (not holding with the PMM and the VBM) is essential: in fact it is equivalent in the HBM to $-\frac{a}{b-1}<1$, ensuring that the graph of $\PPP$ intersects the line $\PPP=P_0$ for $P_0\in\,]0,1[$.

\begin{figure}
\centering
\includegraphics[width=\textwidth]{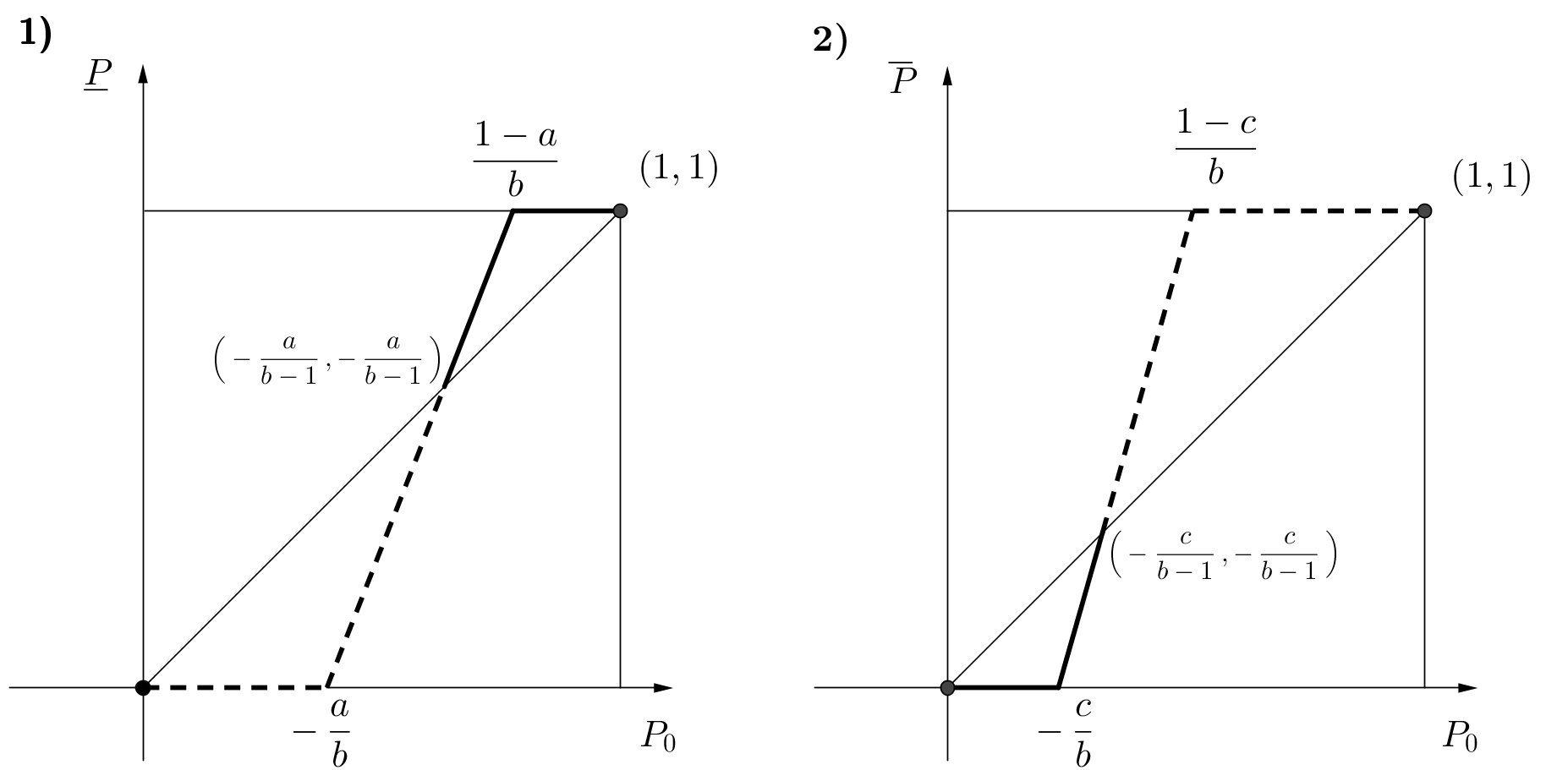}
\caption{Plots of $\PPP,\overline P$ against $P_0$ in the HBM. $\overline P$ and $\PPP$ are dashed bold in their prudential part, continuous bold otherwise. $1)$ $\PPP$ against $P_0$. $2)$ Its conjugate $\overline P$  against $P_0$.}
\label{fig_HBM}
\end{figure}

Turning to $\overline P$, we get to specular conclusions. Again the $\overline P$-assessor is subject to conflicting moods: s/he is unwilling to sell high probability events, but would give away for free low probability events. The lower barrier represents now the imprudent behaviour at its utmost degree. Its length $-\frac{c}{b}$ is smaller or at worst equal to the length $1-\frac{1-c}{b}$ of the upper barrier, because $b+2c\ge 1$ by Lemma \ref{lemma_link_sottosopraNL} $(a)$ (equality holds iff $b+2c=1$). The upper barrier emphasises now the cautious attitude. See also Figure \ref{fig_HBM}, 2) for a graphical illustration. 

From the discussion above, one would probably guess that the HBM can be no more than 2-coherent, thus crediting the partly irrational beliefs it represents with lack of coherence. Perhaps surprisingly, while this is true for the `typical' HBM, coherence is compatible with some HBM. Even more, there are instances of $\PPP$ (and $\overline P$) which are precise probabilities, as in the following example. 

\begin{Ex}
\label{ex_HB_precise}
Given $\PP=\{\omega_1,\omega_2,\omega_3\}$, define the probability $P_0$ as in Table \ref{ex_precise0}. Choosing $b=1.25, \,a=-0.15,$ Equation \eqref{strict_cond_HB} holds, with $b+2a=0.95\neq 1$, while $c=-0.10$. Thus $(\PPP,\overline P)$ is a HBM, and $\PPP,\overline P$ may be computed with \eqref{lower_HBM}, \eqref{upper_HBM}. The results are displayed in Table \ref{ex_precise0}, and $\PPP=\overline P=P$, with $P$ a precise probability (concentrated on $\omega_3$). 

\begin{table}[htbp!]
\begin{center}
\begin{tabular}{c|c|c|c|c|c|c|c|c}
& $\omega_1$ & $\omega_2$ & $\omega_3$ & $\omega_1\vee\omega_2$ & $\omega_1\vee\omega_3$ & $\omega_2\vee\omega_3$ & $\emptyset$ & $\Omega$\\
&&&&&&&&\\[-1em] \hline &&&&&&&&\\[-1em]
$P_0$ & 0.02 & 0.02 & 0.96 & 0.04 & 0.98 & 0.98 & 0 & 1\\
&&&&&&&&\\[-1em] \hline &&&&&&&&\\[-1em]
$\PPP$ & 0 & 0 & 1 & 0 & 1 & 1 & 0 & 1\\
&&&&&&&&\\[-1em] \hline &&&&&&&&\\[-1em]
$\overline P$ & 0 & 0 & 1 & 0 & 1 & 1 & 0 & 1
\end{tabular}
\caption{Values of $P_0,\,\PPP,\,\overline P$ in Example \ref{ex_HB_precise}.}
\label{ex_precise0}
\end{center}
\end{table}
\end{Ex}

In the next subsection, we investigate more closely coherent HBMs.

\subsection{Coherent upper/lower probabilities within HBMs}
\label{coherent_uplow_HB}
In order to study the properties of coherent imprecise probabilities within HBMs, it is convenient to suppose first that $(\PPP,\overline P)$ is defined on $\mathcal A(\PP)$, with $\PP$ a \emph{finite} partition. 

The following proposition is a first result, which is instrumental in the proofs of later propositions:

\begin{Prop}
\label{Psup_sub_implies}
Let $\overline P:\mathcal A(\PP)\to \RR$ be an upper probability satisfying monotonicity and subadditivity (i.e., condition \eqref{subadd}). Suppose that $\PP=\{\omega_1,\dots,\omega_n\}$. 
Then, for any $A\in\mathcal A(\PP)$,
\begin{itemize}
\item[$(a)$] $A\in \mathcal N_{\overline P}$ iff $\displaystyle A=\bigvee_{i=1}^k \omega_{j_i}$, with $\omega_{j_i}\in\mathcal N_{\overline P}$, $i=1,\dots,k$ (if $k=0$, $\displaystyle \bigvee_{i=1}^k \omega_{j_i}=\emptyset$).
\end{itemize}
Suppose further that $\overline P$ is the upper probability of a HBM. Then also:
\begin{itemize}
\item[$(b)$] $A\in\mathcal E_{\overline P}$ iff $\displaystyle A=\omega^+\vee \bigvee_{i=1}^k \omega_{j_i}$, with $\omega^+\in \mathcal E_{\overline P}\, (\omega^+\in\PP), \, \omega_{j_i}\in \mathcal N_{\overline P}, \, P_0(\omega_{j_i})=0$, $i=1,\dots,k$ (if $k=0$, $\displaystyle \bigvee_{i=1}^k \omega_{j_i}=\emptyset$);
\item[$(c)$] if $A\in\mathcal E_{\overline P}$, there exists $\omega^*\in\mathcal E_{\overline P}\cup\mathcal U_{\overline P}$ ($\omega^*\in\PP$), $\omega^*\neq \omega^+$, such that 
\begin{equation}
\label{+*}
bP_0(\omega^+)+ c + bP_0(\omega^*) + c \ge 1.
\end{equation}
\end{itemize}
\end{Prop}

\begin{Cor}
\label{unico_e}
In the hypotheses of Proposition \ref{Psup_sub_implies}, $\forall A\in \mathcal E_{\overline P}$, there exists one and only one $\omega^+\in\PP$ such that $\overline P(A)=\overline P(\omega^+)$. 
\end{Cor}

Although its proof is essentially contained in that of Proposition \ref{Psup_sub_implies}, Corollary \ref{unico_e} highlights an important hindrance of subadditivity, hence also of coherence, of $\overline P$ in the HBM. 

In fact, if $\overline P$ is subadditive (or coherent) in Definition \ref{HBM_mod}, then necessarily \emph{any} event $A$ whose upper probability is non-trivial ($\overline P(A)\in\,]0,1[$) is made of a certain number of atoms of $\PP$, \emph{only one} of which, $\omega^+$, has positive upper probability, and $\overline P(A)=\overline P(\omega^+)$. In other words, the remaining atoms are irrelevant for forming the upper probability $\overline P(A)$. Immediate follow-ups of Proposition \ref{Psup_sub_implies} $(b)$ and Corollary \ref{unico_e} are also, for a subadditive $\overline P$, putting $n\stackrel{\rm def}{=}|\PP|$:
\begin{itemize}
\item[$i)$] The number of distinct values of $\overline P$ on $\mathcal A(\PP)$ is at most $n+2$.
\item[$ii)$] If it is exactly $n+2$, these are the upper probabilities of $\emptyset, \Omega$, and of the $n$ atoms of $\PP$. If $A\in\mathcal A(\PP)\setminus (\PP\cup\{\emptyset\})$, then, since $\NNN_{\overline P}=\{\emptyset\}$, $\overline P(A)=1$, i.e., $\overline P$ is vacuous on all non-atomic events (because any such event is implied by at least two atoms in $\mathcal E_{\overline P}$, and Proposition \ref{Psup_sub_implies} $(b)$ applies). 
\end{itemize}
Note that $ii)$ is a more restrictive constraint than the maxitivity condition $\displaystyle \overline P(A)=$ $\displaystyle\max_{\omega \Rightarrow A}\overline P(\omega)$ satisfied by \emph{possibility measures}, which are already rather special kinds of coherent upper probabilities \cite[Section 4.6.1]{book}. This does not, however, imply that any coherent $\overline P$ in a HBM  is a possibility measure: in fact, possibilities require that there exists $\omega^*\in\PP$ such that $\overline P(\omega^*)=1$, while $\overline P$ mail fail to satisfy this condition.

A less immediate consequence of subadditivity is a steepening effect on the parameter $b$, implied by the following result.

\begin{Prop}
\label{cond_sub_HB}
In a HBM, let $\overline P:\mathcal A(\PP)\to\RR$ be subadditive. If $\PP$ is finite and $|\mathcal E_{\overline P}\cap \PP|\ge 2m>0$, then
\begin{equation}
\label{bound_b}
b>\max\bigg\{m,\max_{\substack{\omega_i,\omega_j\in\mathcal E_{\overline P}\\i\neq j}}\frac{1}{\overline P(\omega_i) + \overline P(\omega_j)}\bigg\}.
\end{equation}
\end{Prop}

Thus, a higher number of atoms in $\mathcal E_{\overline P}$ requires a larger $b$ in a subadditive $\overline P$. In Figure \ref{fig_HBM}, $2)$, this makes steeper the non-horizontal plot of $\overline P$ against $P_0$, a segment on the line $\overline P=bP_0 + c$. 
On its turn, a larger $b$ increases the width $-\frac{c}{b}=\frac{a+b-1}{b}$ of the lower barrier, which seems contradictory, since the barrier is in the risk-seeking area for $\overline P$. This fact is counterbalanced by the increase in the number of events of $\mathcal A(\PP)$ in the risk-averse area (because they are given $\overline P=1$).

The restrictions pointed out for a coherent $\overline P$ in the HBM all depend on its subadditivity. As a matter of fact, subadditivity turns out to be equivalent to coherence of $\overline P$, even with arbitrary (infinite) partitions $\PP$. 
\begin{Prop}
\label{characterise_coherence}
Let $\overline P:\mathcal A(\PP)\to\RR$ be the upper probability of a HBM, with $\PP$ an arbitrary (finite or not) partition. Then $\overline P$ is coherent if and only if it is subadditive.
\end{Prop}

%

It is easy to check that $\overline P$ is subadditive if and only if its conjugate $\PPP$ satisfies property \eqref{quasi_superadd}. Thus Proposition \ref{characterise_coherence} has the following, slightly less straightforward, correspondent for lower probabilities:

\begin{Prop}
\label{characterise_lower_coherence_gen}
A lower probability $\PPP:\mathcal A(\PP)\to\RR$ in a HBM is coherent if and only if $\PPP(A) + \PPP(B)\le 1+ \PPP(A\wedge B)$, $\forall A,B\in \mathcal A(\PP)$.
\end{Prop}

Interestingly, even though superadditivity is not, so to say, the conjugate property for lower probabilities of subadditivity, it nevertheless holds for any lower probability in the HBM, even those that are 2-coherent but not coherent:

\begin{Prop}
\label{Pinf_HB_super}
Given a HBM, its lower probability $\PPP$ is superadditive, i.e., satisfies \eqref{superadd}.
\end{Prop}

It is easy to see that the conjugate $\overline P$ of a superadditive $\PPP$ satisfies the following property \cite[Section 2.7.4, $(f)$]{W}:
\begin{equation}
\label{superadd_P sup}
\overline P(A) + \overline P(B)\ge 1+\overline P(A\wedge  B), \quad \text{if }A\vee B=\Omega.
\end{equation}
Therefore, the result corresponding to Proposition \ref{Pinf_HB_super} holds:
\begin{Prop}
In a HBM, its upper probability $\overline P$ satisfies property \eqref{superadd_P sup}.
\end{Prop}

Finally, we prove the following remarkable result for coherent lower/upper probabilities in the HBM:

\begin{Prop}
\label{HBM_coh_implies_alter}
If, in a HBM, $\overline P$ ($\PPP$) is coherent, then it is 2-alternating (2-monotone).
\end{Prop}

\subsection{Precise probabilities within HBMs}
\label{section_precise_HBM}

The HBM is the only NL submodel offering a certain variety of precise probabilities as its special cases. In Example \ref{ex_HB_precise}, the assumptions of Proposition \ref{equiv_sotto=sopra} $(b)$ are satisfied, and $P$ is in fact 0-1 valued. Furthermore, it is also a precise probability.

If $P_0$ is 0-1 valued, the HBM transforms $P_0$ into itself, i.e., $\PPP=\overline P=P_0$. This ensues from Proposition \ref{link_HB_P0}: taking for instance $\PPP$, from $jj)$, $jjj)$ we have $\mathcal A(\PP)=\NNN_{P_0}\cup \UUU_{P_0}\subset \NNN_{\PPP}\cup \UUU_{\PPP}$, implying $\EEE_{\PPP}=\emptyset$ (note that with the VBM instead, in general the inclusion $\NNN_{P_0}\subset \NNN_{\PPP}$ is false). 

We remark that there are 0-1 valued lower/upper probabilities in the HBM that are not precise probabilities, when $b+2a\neq 1$, and that may or may not be coherent, as in the next example. 

\begin{Ex}
\label{01_not_coherent}
Let $\PP=\{\omega_1,\omega_2,\omega_3\}$. In Table \ref{table_01_not_coherent}, $\PPP=\overline P$ is obtained from the uniform probability $P_0$ by \eqref{lower_HBM}, \eqref{upper_HBM}, with $a=-4,\, b=8.5$. $\overline P$ (hence $\PPP$) is not coherent by Proposition \ref{characterise_coherence}, being not subadditive: $\overline P(\omega_1)+\overline P(\omega_2)<\overline P(\omega_1\vee\omega_2)$.

In the same Table \ref{table_01_not_coherent}, $\PPP'$ and its conjugate $\overline P'$ are given by \eqref{lower_HBM}, \eqref{upper_HBM} from $P_0'$, with $a=-4,\,b=6$. Using Propositions \ref{characterise_coherence} or \ref{characterise_lower_coherence_gen}, it may be checked that $\PPP',\,\overline P'$ are coherent (but they are not probability measures).

\begin{table}[htbp!]
\begin{center}
\begin{tabular}{c|c|c|c|c|c|c|c|c}
&&&&&&&\\[-1em] & $\omega_1$ & $\omega_2$ & $\omega_3$ & $\omega_1\vee\omega_2$ & $\omega_1\vee\omega_3$ & $\omega_2\vee\omega_3$ & $\emptyset$ & $\Omega$\\
&&&&&&&\\[-1em] \hline &&&&&&&\\[-1em]
$P_0$ & $\frac{1}{3}$ & $\frac{1}{3}$  & $\frac{1}{3}$ & $\frac{2}{3}$  & $\frac{2}{3}$ & $\frac{2}{3}$ & 0 & 1\\
&&&&&&&\\[-1em] \hline &&&&&&&\\[-1em]
$\PPP=\overline P$ & 0 & 0 & 0 & 1 & 1 & 1 & 0 & 1\\
&&&&&&&\\[-1em] \hline &&&&&&&\\[-1em]
$P_0'$ & $\frac{1}{2}$ & $\frac{29}{60}$  & $\frac{1}{60}$ & $\frac{59}{60}$  & $\frac{31}{60}$ & $\frac{1}{2}$ & 0 & 1\\
&&&&&&&\\[-1em] \hline &&&&&&&\\[-1em]
$\PPP'$ & 0 & 0 & 0 & 1 & 0 & 0 & 0 & 1\\
&&&&&&&\\[-1em]\hline &&&&&&&\\[-1em]
$\overline P'$ & 1 & 1 & 0 & 1 & 1 & 1 & 0 & 1
\end{tabular}
\caption{Values for Example \ref{01_not_coherent}.}
\label{table_01_not_coherent}
\end{center}
\end{table}
\end{Ex}

There are also precise probabilities within the HBM that are not necessarily 0-1 valued, if  $b+2a=1$. These measures must satisfy severe constraints. 

A general question is when $\PPP,\overline P$ in a HBM coincide and are a precise probability. The next result gives an answer in the finite case, showing explicitly how special this situation is. The later Proposition \ref{N+1} is more general, but less meaningful in letting us evaluate the rarity of the precise probabilities.

\begin{Prop}
\label{prop_N}
Let $\PPP,\overline P$ in a HBM  be coherent on $\mathcal A(\PP)$, $\PP$ a finite partition. Then $\PPP=\overline P=P(\neq P_0)$, $P$ probability, if and only if one of the following holds:
\begin{itemize}
    \item[$(a)$] $\exists \omega^+\in\PP:$ $\PPP(\omega^+)=\overline P(\omega^+)=1$, and $\PPP(\omega)=\overline P(\omega)=0$, $\forall \omega\in\PP\setminus\{\omega^+\}$;
    \item[$(b)$] $a=c<0$, and $|\mathcal E_P\cap \PP|=2$.
\end{itemize}
\end{Prop}

It is interesting to point out that while a coherent $\overline P$ in a HBM may assume at most $|\PP|+2$ distinct values on $\mathcal A(\PP)$, as seen in Section \ref{coherent_uplow_HB}, Proposition \ref{prop_N} implies that $\overline P$, as a probability, may take at most 4 values. This gives a clear idea of how restrictive precision is within HBMs.

In Example \ref{ex_HB_precise} we have already encountered a HBM  probability $P$ meeting case $(a)$ of Proposition \ref{prop_N}. In the next example $P$ satisfies case $(b)$ of the same Proposition \ref{prop_N}.


\begin{Ex}
\label{dF}
Let again $\PP=\{\omega_1,\omega_2,\omega_3\}$. Take $a=-0.125,\,b=1.25$ (hence $b+2a=1$), and define 
$$
P_0(\omega_1)=0.3,\quad P_0(\omega_2)=0.7,\quad P_0(\omega_3)=0.
$$
Using \eqref{lower_HBM} or \eqref{upper_HBM}, we obtain $(\PPP=\overline P=)P$ given by:
\begin{align*}
P(\omega_1) =P(\omega_1\vee\omega_3) & =0.25, &  
P(\omega_2) =P(\omega_2\vee\omega_3) & =0.75, \\ 
P(\omega_1\vee\omega_2) & =P(\Omega)=1, & 
P(\omega_3) & = 0.
\end{align*}
$P$ is a precise probability.

\end{Ex}

\begin{Prop}
\label{N+1}
In a HBM, it is $\PPP=\overline P=P$, $P$ a probability measure, if and only if, $\forall A,B\in\mathcal A(\PP)$,
\begin{itemize}
    \item[$(a)$] $\PPP(A\vee B)\le \PPP(A) + \PPP(B)$;
    \item[$(b)$] $\PPP(A)+\PPP(B)\le 1 + \PPP(A\wedge B)$.
\end{itemize}
\end{Prop}

\section{Other Nearly-Linear Models}

We investigate in this section the third NL model arising from the classification of Table \ref{NLmodels}, and discuss an extreme model that can be accommodated into the framework of NL models. 

\subsection{The Restricted Range Model}
\label{sub_RRM}
A look at Table \ref{NLmodels} suggests that there is just one family of lower probabilities still missing in our analysis of NL models, identified by the parameter constraints of Case 6 (roughly: some strict inequalities might be relaxed, admitting the equality). Correspondingly, from the starting point of the PMM,  we have so far seen the cases $a+b\le 1$ and $a+b\ge 1$, while keeping the PMM constraint $a\le 0$. Thus, it remains to explore the case of a positive $a$. From these considerations, we introduce in this section $\mu$, a NL$(a,b)$ measure, with the constraints 
\begin{equation}
\label{ab_RRM}
b+2a\le 1, \quad a>0.
\end{equation}
In principle, $a=0$ could be allowed too. This corresponds to the lower probability of the $\varepsilon$-contamination model, included also in the VBM. We prefer to rule out this case, to avoid overlapping with the VBM.

From inequalities \eqref{ab_RRM}, it follows straightforwardly that
\begin{equation*}
a+b< 1, \quad b<1.
\end{equation*}
Clearly, $\mu$ is a lower probability, and since $0<a\le bP_0(A) + a\le b+a<1$, Equation \eqref{def_mu_NL} 
simplifies to 
$$
\mu(A)=bP_0(A) + a, \quad \forall A\in \mathcal A(\PP)\setminus\{\emptyset,\Omega\}.
$$
Using \eqref{c} for the conjugate $\mu^c$, we have:
\begin{Def}
A \emph{Restricted Range Model (RRM)} is a NL Model with $\PPP$ and its conjugate $\overline P$ given by, $\forall A\in \mathcal A(\PP)\setminus\{\emptyset,\Omega\}$,
\begin{align}
\label{lower_RRM}
\PPP(A) & =bP_0(A) + a,\\
\label{upper_RRM}
\overline P(A) & = bP_0(A) +c,
\end{align}
with $a,b$ satisfying \eqref{ab_RRM}, $c$ given by \eqref{c}, and $\PPP(\emptyset)=\overline P(\emptyset)=0$, $\PPP(\Omega)=\overline P(\Omega)=1$.
\end{Def}

The consistency properties of a RRM  are easy to fix:

\begin{Prop}
\label{consistency_RRM}
Let $(\PPP,\overline P)$ be a RRM. Then $\PPP,\overline P$ are 2-coherent, and coherent iff $|\PP|=2$.
\end{Prop}

Thus, a RRM is 2-coherent, but typically not coherent. Proposition \ref{consistency_RRM} lets us deduce also the following interesting result. 

\begin{Prop}
	\label{coherence_implies_2monotonicity}
In a NL model, if $\PPP$ ($\overline P$) is coherent, then it is 2-monotone (2-alternating).
\end{Prop}

To clarify which beliefs may be supported by a RRM, we state its following elementary properties:

\begin{Prop}
\label{properties}
Let $(\PPP,\overline P)$ be a RRM, $A\in \mathcal A(\PP)\setminus\{\emptyset,\Omega\}$. Then, concerning $\PPP$,
\begin{itemize}
\item[$k)$] $\PPP(A)>P_0(A)$ iff $P_0(A)<\frac{a}{1-b}$;
\item[$kk)$] $\PPP(A)\in [a,a+b]$.
\end{itemize}
As for $\overline P,$
\begin{itemize}
\item[$k')$] $\overline P(A)<P_0(A)$ iff $P_0(A)>\frac{c}{1-b}$;
\item[$kk')$] $\overline P(A)\in[c,b+c]=[1-(a+b),1-a]$.
\end{itemize}
\end{Prop}

Properties $kk)$, $kk')$ justify the name Restricted Range Model: the values of $\PPP,\overline P$ on non-trivial events belong to proper subsets of the interval $[0,1]$, distinct unless $\PPP=\overline P$, but both having the same width $b$. Thus $b$ measures the range of the admissible $\PPP$- and $\overline P$-evaluations for any non-trivial event.
Now take $\PPP$, whose plot against $P_0$ is illustrated in Figure \ref{fig_RRM}, 1). 
\begin{figure}[htbp!]
\centering
\includegraphics[width=\textwidth]{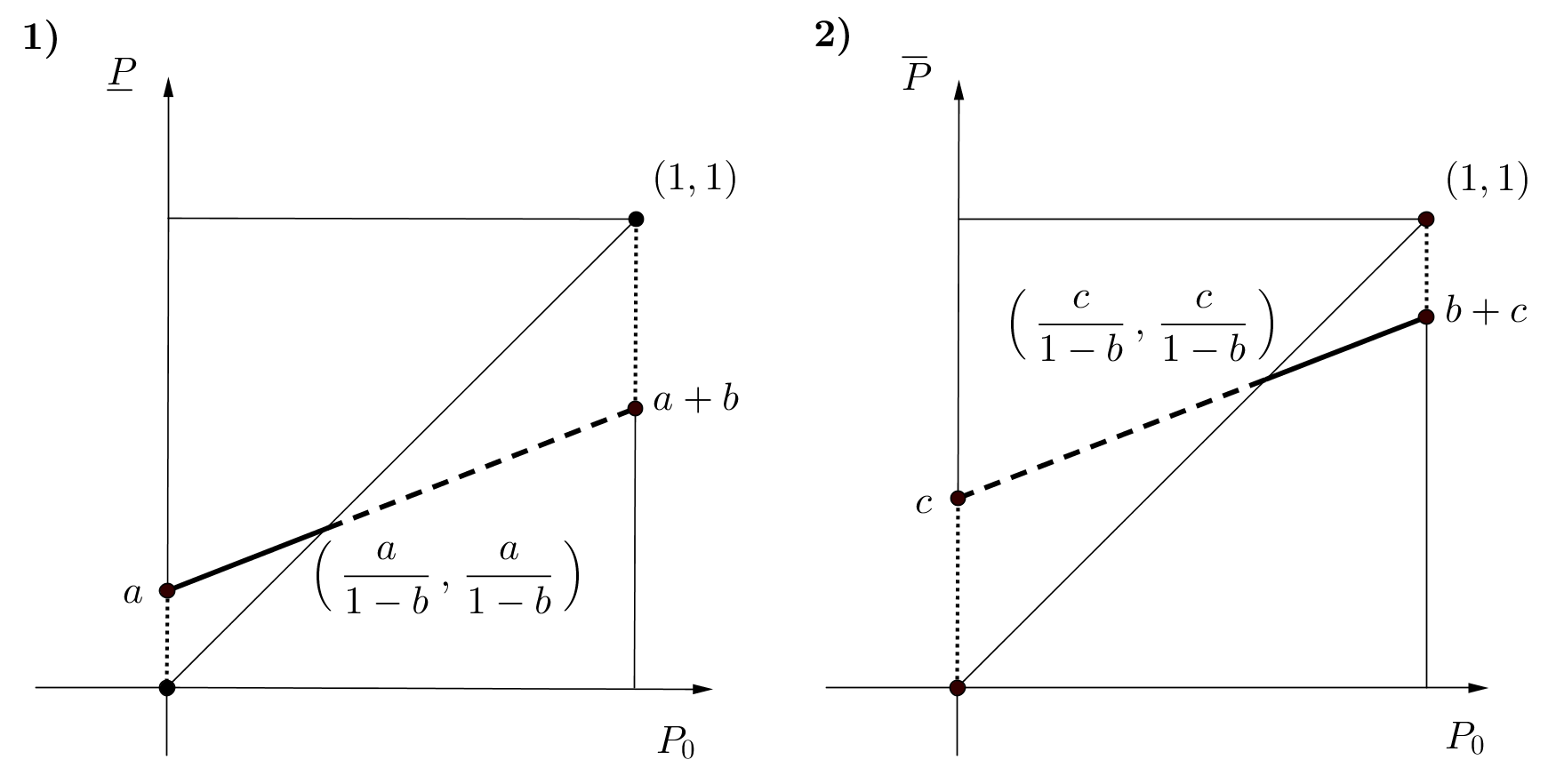}
\caption{Plots of $\PPP,\overline P$ against $P_0$ in the RRM. $\PPP$ and $\overline P$ are dashed bold in their prudential part, continuous bold elsewhere. $1)$ $\PPP$ against $P_0$. $2)$ Its conjugate $\overline P$ against $P_0$.}
\label{fig_RRM}
\end{figure}
By property $k)$, the evaluation of a $\PPP$-assessor is imprudent for low probability events (low meaning $P_0<\frac{a}{1-b}$), prudential otherwise. Interpreting again $P_0(A)$ as the `true' probability for event $A$, the $\PPP$-assessor overestimates low probability events, while underestimates high probability event, which s/he buys for no more than $a+b<1$. Here $a$ measures the assessor's imprudence: her/his buying price for any non-impossible event is at least $a$.

So far, the RRM looks  very similar to the HBM, as both can represent an assessor's conflicting attitudes. At a closer look, a difference is that the $\PPP$-assessor is imprudent with high probability events, using the HBM, with low probability events, by the RRM. Further, in the prudential area, the agent is more cautious with the HBM: there are events (those $A$ such that $P_0(A)<-\frac{a}{b}$) which s/he does not buy, unless they are given away for nothing. On the contrary, in the RRM the assessor is willing to buy \emph{any} event paying at least $a$. The opposite holds in the imprudent area. 

Probabilities play a negligible role within RRMs: as a consequence of Proposition \ref{consistency_RRM}, $\PPP=\overline P=P$ is a probability iff $|\PP|=2$.

In the whole, we may say that the RRM is the least consistent among the NL models, and (correspondingly) the one which deviates more from the PMM.

\subsection
{Degenerate NL models}
\label{HZ}

Define $\mu_h$ as $\mu_h(A)=a\in[0,1]$, $\forall A\in\mathcal A(\PP)\setminus\{\emptyset,\Omega\}$, and $\mu_h(\emptyset)=0,\mu_h(\Omega)=1$. This is a special \emph{Hurwicz capacity}, following \cite{CEG}.

So far, we supposed $b>0$ in Definition \ref{NL_mod} of a NL model; this assumption has been useful in previous computations, to prevent dividing by zero. However, allowing $b=0$, the Hurwicz capacity $\mu_h$ can be formally computed by Equation \eqref{def_mu_NL}, as well as its conjugate $\mu_h^c$ given by $\mu_h^c(A)=1-a$, $\forall A\in\mathcal A(\PP)\setminus\{\emptyset,\Omega\}$. 

Thinking of the couple $(\mu_h,\mu_h^c)$, we term $a$ the smaller between $\mu_h(A)$, $\mu_h^c(A)$, i.e., with the notation used so far, $a\le 1-a$. Consequently, $\mu_h\le\mu_h^c$, so $\mu_h$ ($\mu_h^c$) is a lower (upper) probability. Summing up:
\begin{Def}
A \emph{Degenerate NL Model} is a couple $(\PPP_h,\overline P_h)$, with $\PPP_h$ NL$(a,0)$, $\overline P_h$ NL$(1-a,0)$, $a\in [0,\frac{1}{2}]$.
\end{Def}

Despite its simplicity, a degenerate NL model may be compatible with different degrees of consistency, as the following proposition points out referring to $\PPP_h$ (conjugacy ensures the same properties for $\overline P_h$). 
\begin{Prop}
\label{prop_deg_mod}
Let $(\PPP_h,\overline P_h)$ be a degenerate NL model. 
\begin{itemize}
\item[$(a)$] $\PPP_h:\mathcal A(\PP)\to\RR$ is 2-coherent.
\item[$(b)$] The restriction of $\PPP_h$ on $\mathcal A(\PP)\setminus\{\emptyset,\Omega\}$ is convex.
\item[$(c)$] If $\PP$ is finite, $\PP=\{\omega_1,\dots,\omega_n\}$, $\PPP_h$ is C-convex on $\mathcal A(\PP)$ iff $a\le \frac{1}{n}$ iff $\PPP_h$ avoids sure loss on $\mathcal A(\PP)$.
\end{itemize}
\end{Prop}

Behaviourally, $\PPP_h$ expresses willingness to buy any non-trivial event in $\mathcal A(\PP)$ at the same price. If $a=0$, $\PPP_h$ coincides with the vacuous lower probability and is coherent. In general, $\PPP_h$ is not coherent: already its restriction on $\PP$ is incoherent, if $\PP$ is infinite and $a>0$.

Still referring to the restriction of $\PPP_h$ on $\PP$, note that it expresses a symmetry judgement over the atoms of $\PP$. With coherence, this may be acceptable for $a>0$ only if $\PP$ is finite, and then with the constraint $a\le \frac{1}{|\PP|}$. With precise probabilities, it is even required that $a=\frac{1}{|\PP|}$;  this corresponds to the uniform probability. By Proposition \ref{prop_deg_mod}, 2-coherence and convexity are instead compatible with a symmetry or indifference judgement on even larger sets: the whole $\mathcal A(\PP)$ with 2-coherence, $\mathcal A(\PP)\setminus\{\emptyset,\Omega\}$ with convexity, again $\mathcal A(\PP)$, but only if $\PP$ is finite, assuming C-convexity (or equivalently the condition of avoiding sure loss).
\begin{Rem}
In general, neither between 2-coherence and convexity implies the other one, while a lower prevision \emph{defined on a real vector space} of gambles is 2-coherent and convex if and only if it is coherent \cite[Proposition 13]{PV16}. The results in this section show that this equivalence is not general. In fact, a Hurwicz capacity $\PPP_h$ defined on $\PP$ (or also on $\mathcal A(\PP)\setminus\{\emptyset,\Omega\}$) is both 2-coherent and convex, by Proposition \ref{prop_deg_mod} $(a)$, $(b)$, but not coherent if $a>0$ and $\PP$ is infinite.   
\end{Rem}

\section{Comparisons with Other Models}
\label{sec_compare}
In this section we compare NL models with other models proposed in the literature that are to a certain extent overlapping with them. Prior to this, we discuss in the next subsection how the properties of the sets $\NNN_{\mu},\UUU_{\mu}$ of, respectively, null and universal events for a measure $\mu$ may vary with the degree of consistency required to $\mu$. The topic is of interest in itself, although our primary aim is to exploit it while comparing NL models and neo-additive capacities in Section \ref{NL_neo}.

\subsection{Null and universal events with different consistency notions} 
\label{NU}
Let $\mu$ be an uncertainty measure on $\mathcal A(\PP)$. Recall that a set of events $\mathcal F\subset \mathcal A(\PP)$ is a \emph{filter} if:
$$
\Omega\in\mathcal F, \, \emptyset \notin\mathcal F; \quad A,B\in\mathcal F \Rightarrow A\wedge B\in\mathcal F; \quad \text{if }A\in \mathcal F, A\Rightarrow B \text{ then } B\in\mathcal F.
$$
A set of events $\mathcal I\subset \mathcal A(\PP)$ is an \emph{ideal} if:
$$
\emptyset \in\mathcal I, \, \Omega\notin\mathcal I; \quad A,B\in\mathcal I \Rightarrow A\vee B\in\mathcal I; \quad \text{if }A\in \mathcal I, B\Rightarrow A \text{ then } B\in\mathcal I.
$$
When $\mu=P$, a probability, it is well-known that:
\begin{itemize}
\item[$i)$] $\UUU_P$ is a filter, while $\NNN_P$ is an ideal.
\item[$ii)$] $A\in\UUU_P$ iff $\neg A\in\NNN_P$.
\end{itemize}
When $\mu$ is an imprecise probability, these properties are only partly retained, and the way $\mu$ deviates from them depends on its exact type of consistency. We present significant cases for lower probabilities in the next result.

\begin{Prop}
	\label{ideals_and_others}
Let $\PPP:\mathcal A(\PP)\to\RR$ be a lower probability.
\begin{itemize}
\item[$(a)$] If $\PPP$ is coherent, then $\UUU_{\PPP}$ is a filter, but $\NNN_{\PPP}$ may not be an ideal. Further, it holds that 
\begin{equation}
\label{U_implies_N}
A\in\UUU_{\PPP}\Rightarrow \neg A\in\NNN_{\PPP}.
\end{equation}
\item[$(b)$] If $\PPP$ is 2-coherent, $\UUU_{\PPP}$ may not be a filter: in particular, it is possible that $\PPP(A)=\PPP(B)=1$, $A\wedge B\neq \emptyset$, but $\PPP(A\wedge B)=0$. Equation \eqref{U_implies_N} holds.
\item[$(c)$] If $\PPP$ is convex, $\UUU_{\PPP}$ may not be a filter, but if $\PPP(A)=\PPP(B)=1$, $A\wedge B\neq \emptyset$, then $\PPP(A\wedge B)\ge \frac{1}{2}$ (the bound is tight). Equation \eqref{U_implies_N} does not necessarily hold; it does if $\PPP$ is C-convex.
\end{itemize}
\end{Prop}

We may conclude that, in particular, condition $A\in\UUU_{\mu}\Leftrightarrow \neg A\in\NNN_{\mu}$ cannot be expected to hold with imprecise probabilities, while its weakened form \eqref{U_implies_N} does for several common consistency requirements. 


\subsection{NL models and other capacities}
\label{NL_neo}
Since $\PPP,\overline P$ in any NL model are (also) capacities, NL models overlap to some extent with certain models exploiting specific capacities for various purposes.

To the best of our knowledge, the oldest reference of this kind is a paper by Rieder \cite{R} focused on statistical robustness. In \cite{R}, Rieder introduces a specific VBM and proves the 2-monotonicity of its $\underline P$. His model is a special case of ours, since he requires (using our parametrisation) the extra condition $a\ge -1$.

The comparison is more complex with \emph{neo-additive capacities} (NACs)\footnote{Neo-additive is the acronym for Non-Extreme Outcomes additive.} introduced in \cite{CEG}, and \emph{Generalised neo-additive capacities} (GNACs) introduced in \cite{EGL12}. These types of capacities, originally devised within utility theory, were also studied and characterised in an abstract measure setting in \cite{GH16}. 
In these papers, a probability $P_0$ is given, and the capacity $\mu$ is a linear affine transformation of $P_0$ on the set $\EEE$ of essential events, i.e., $\mu(A)=bP_0(A) +a $, $\forall A\in\EEE$. However, with NL models we derive the sets $\NNN_{\mu},\UUU_{\mu}$ \emph{a posteriori}, after defining $\mu$ by \eqref{def_mu_NL} and using \eqref{N}, \eqref{U}, while $\mathcal E_{\mu}$ is the difference $\mathcal A(\PP)\setminus(\mathcal N_{\mu}\cup \mathcal U_{\mu})$. By contrast,
in this approach a set $\NNN$ of events which \emph{must} have $\mu$-measure 0 is fixed \emph{a priori}, requiring further that $\NNN$ is an ideal and that 
\begin{equation}
\label{N_implies_U}
A\in\UUU\Leftrightarrow \neg A\in\NNN.
\end{equation}
Then $\mu$ is defined on $\mathcal A(\PP)$ by:
\begin{equation}
\label{neo}
\mu(A)=
\begin{cases}
0 & \text{if }A\in\NNN\\
bP_0(A) + a & \text{if }A\in\EEE\stackrel{\rm def}{=}\mathcal A(\PP)\setminus(\NNN\cup\UUU)\\
1 &\text{if } A\in\UUU
\end{cases}
.
\end{equation}
Thus, $\mathcal N$ does not necessarily coincide with what we defined as $\mathcal N_{\mu}$: there may be $A\in\mathcal E$ having $\mu(A)=0$ (if $bP_0(A) + a=0$).

Unlike NL models, neo-additive capacities require that $a,b\in[0,1]$,\footnote{This is seen after a relabelling in \cite[Definition 3.2]{CEG}, originally defining NACs as a mixture of $P_0$ and a Hurwicz capacity.} $0\le bP_0(A) + a\le 1$\, $\forall A\in\EEE$, and that $\NNN_{P_0}\supset \NNN$. GNACs require $b\ge 0$, $a\in\RR$ as in the NL models, but with the additional condition that $\PP$ is finite. Also, the set $\NNN$ still has to be an ideal, and might be determined on the basis of utility theory considerations.

In general, the sets $\NNN_{\PPP},\NNN_{\overline P}$ in the NL models cannot be taken as $\NNN$: they are typically not ideals, 
while $\mathcal N_{\PPP},\mathcal U_{\PPP}$ may not satisfy \eqref{N_implies_U}. For them the implication \eqref{U_implies_N} holds instead 
under coherence or 2-coherence, as we have seen in Section \ref{NU}. 
Nor can $\EEE$ be identified with $\EEE_{\PPP}$ or $\EEE_{\overline P}$: it is possible in \eqref{neo} that $A\in\EEE$ and $\mu(A)=0$ or $\mu(A)=1$. Note also that, if one wishes for whatever reason to impose that certain events belong to $\NNN$, then most NL models are unfit for this purpose. For instance, take $\PPP_{\rm PMM}$: here $\UUU_{\PPP_{\rm PMM}}=\{\Omega\}$, hence to satisfy \eqref{N_implies_U} and \eqref{neo} it is necessary that $\NNN=\{\emptyset\}$. Similarly with the RRM.

There are further differences with NL models:
\begin{itemize}
\item[-] It is not distinguished whether any $\mu$ is a lower or an upper probability.
\item[-] The consistency properties of NACs and GNACs (coherence, 2-coherence) are not investigated, even though some special cases of GNACs are studied in \cite{EGL12}. 2-monotone or 2-alternating capacities (termed convex or concave, respectively) are especially considered.
\end{itemize}

An analogy lies instead in the motivations for introducing these models, which are to represent coexistence of both optimism and pessimism towards uncertainty or ambiguity. This is closed to the interpretation in terms of the assessor's prudential or non-prudential attitude in NL models.

\subsection{NL models and probability intervals}
\label{sec_intervals}
Given a \emph{finite} partition $\PP=\{\omega_1,\dots, \omega_n\}$, a \emph{probability interval} $\II=(l,u)$ on $\PP$ is an $n$-tuple of intervals $[l_i,u_i]$, $0\le l_i\le u_i\le 1$, $i=1,\dots,n$ \cite{CHM94}.

It may be viewed 
\begin{itemize}
\item[$i)$] as a lower and upper probability assignment on $\PP$, 
or equivalently 
\item[$ii)$] as a lower probability assignment on $\PP\cup\{A_i:\neg A_i\in \PP\}$, where, for $i=1,\dots,n$, $\PPP(\omega_i)=l_i$, 
$\PPP(A_i)=1-\overline P(\neg A_i)=1-u_i$ if $\neg A_i=\omega_i$. 
\end{itemize}
It is well-known \cite{book,MMD18,CHM94} that a probability interval $\II$ is coherent (reachable in the terminology of \cite{CHM94}) on $\PP$ iff
$$
u_i + \sum_{j\neq i}l_j \le 1, \quad l_i + \sum_{j\neq i} u_j\ge 1, \quad i=1,\dots,n,
$$
and that if $\II$ is coherent on $\PP$ it has a \emph{least-committal} extension (or natural extension \cite{W}) on $\mathcal A(\PP)$ given by, $\forall A\in\mathcal A(\PP)$,
\begin{align}
\label{extended_lower}
l(A) & = \max\bigg\{\sum_{\omega_i\Rightarrow A} l_i, 1- \sum_{\omega_i\Rightarrow \neg A} u_i\bigg\}, \\
\label{extended_upper}
u(A) & = 1-l(\neg A)= \min\bigg\{\sum_{\omega_i\Rightarrow A} u_i, 1- \sum_{\omega_i\Rightarrow \neg A} l_i\bigg\}.
\end{align}
Moreover, $l(\cdot)\, (u(\cdot))$ is coherent and 2-monotone (2-alternating) on $\mathcal A(\PP)$. Least-committal means that if $\PPP$ is any coherent extension of $\II$ on $\mathcal A(\PP)$ and $\overline P$ is its conjugate, then $\forall A\in\mathcal A(\PP)$
\begin{equation}
\label{N2}
(0\le)\, l(A)\le\PPP(A)\le\overline P(A)\le u(A)\, (\le 1).
\end{equation}
For a better understanding, we call \emph{extended probability interval}, $\II_E$, the least-committal extension \eqref{extended_lower}, \eqref{extended_upper} on $\mathcal A(\PP)$ of a coherent $\II$.


A reason for comparing NL models and probability intervals is that, as has recently been proven \cite{MMD18}, any PMM in a finite setting is an extended probability interval. It is therefore of interest to establish whether this property holds for more general NL models.

We start with VBMs. It appears already from the following result that they are usually not extended probability intervals.


\begin{Prop}
\label{characterisation_extended}
Let $(\PPP,\overline P)$ be a VBM  on $\mathcal A(\PP)$ (Definition \ref{VBM}), $|\PP|=n$, and define $\II=(\PPP|_{\PP},\overline P|_{\PP})$. Then $(\PPP,\overline P)$ is the extended probability interval $\II_E$ of $\II$ if and only if one of the following holds:
\begin{itemize}
\item[$(a)$] $a=0$,
\item[$(b)$] $a+b=1$,
\item[$(c)$] $\forall A\in\mathcal A(\PP)\setminus\{\emptyset,\Omega\}$, $A$ satisfies one of the following conditions (the condition may vary with $A$):
\begin{itemize}
    \item[$(c1)$] $\PPP(A)=0$,
    \item[$(c2)$] $a<0$, $\PPP(A)>0$, $\displaystyle A=\omega^+\vee \bigvee_{j=1}^k \omega_{i_j}$, $k\in\{0,\dots,n-1\}$, $\PPP(\omega^+)>0$, $P_0(\omega_{i_j})=0$, $j=1,\dots,k$,
    \item[$(c3)$] $\PPP(A)>0$ and $\neg A\in\PP$.
\end{itemize}
\end{itemize}
\end{Prop}

We may restate Proposition \ref{characterisation_extended} in a more expressive way, on the grounds of the following considerations. 

Its conditions $(a)$ and $(b)$ require, respectively, that $(\PPP,\overline P)$ is an $\varepsilon$-contami\-nation model or a Pari-Mutuel model.

As for condition $(c)$, it is easy to realise that it is always satisfied when $|\PP|=n\le 3$. 

When $n>3$, $(c)$ can only hold if 
\begin{equation}
\label{ppiu}
|\mathcal P^+|=|\{\omega\in\PP: \PPP(\omega)>0\}|\le 1.
\end{equation}
In fact, suppose that $(c)$ holds, let $\PPP(\omega^+)>0$, and take any $\omega\neq \omega^+$. Then $\PPP(\omega)=0$, or else $A=\omega^+\vee \omega$ satisfies
none of the subcriteria $(c1)$, $(c2)$, $(c3)$. Since $\omega$ is arbitrary in $\PP\setminus\{\omega^+\}$, \eqref{ppiu} holds.

Let us see in detail the two possible situations compatible with \eqref{ppiu}.

\begin{itemize}
\item $|\mathcal P^+|=1$. Then $\exists!\ \omega^+:\PPP(\omega^+)>0$. Taking any $\omega\neq \omega^+$, the event $A=\omega^+\vee \omega$ can only satisfy $(c2)$, but this requires $P_0(\omega)=0$. Thus $P_0$ must be concentrated in $\omega^+$ ($P_0(\omega^+)=1$), which is on the other hand also sufficient for $(c)$.

In fact, since $P_0\ge \PPP$, it must be $\PPP(A)=0$ for any $A$ such that $\omega^+\wedge A=\emptyset$. Therefore, if $\omega^+\wedge A=\emptyset$, $A$ satisfies $(c1)$, while if $\omega^+\wedge A=\omega^+$, $A$ satisfies $(c2)$ (if $\omega\Rightarrow A$, $\omega\neq \omega^+$, $\PPP(\omega)=P_0(\omega)=0$).
 \item 
$|\mathcal P^+|=0$. Then, $(c2)$ does not apply. To satisfy $(c)$, either $\PPP$ is the vacuous lower probability, or those events $A$ with $\PPP(A)>0$ are made of $n-1$ atoms of $\PP$. 
 \end{itemize}
 
 Summarising, we have 
 
 \begin{Prop}
 \label{prop_characterise_intervals}
 Let $(\PPP,\overline P)$ be a VBM on $\mathcal A(\PP)$. $(\PPP,\overline P)$ is the extended probability interval of its restriction on $\PP$ iff one of the following holds:
 \begin{itemize}
     \item[$(a)$] $(\PPP,\overline P)$ is an $\varepsilon$-contamination model,
     \item[$(b)$] $(\PPP,\overline P)$ is a PMM,
     \item[$(c)$] $|\PP|\le 3$,
     \item[$(d)$] $|\PP|>3$ and one the following holds: 
     \begin{itemize}
         \item[$\bullet$] $P_0$ is concentrated on one atom of $\PP$, 
         \item[$\bullet$] $(\PPP,\overline P)$ is the vacuous imprecise probability, 
         \item[$\bullet$] if $\PPP(A)>0$ then  $A$ is made of $n - 1$ atoms.  
     \end{itemize}
 \end{itemize}
\end{Prop}
Turning to the other NL models, there is little to say about RRMs: they are extended probability intervals iff $|\PP|=2$. As for HBMs, when they are coherent (which, as seen in Section \ref{coherent_uplow_HB}, is not the rule), they are also extended probability intervals:

\begin{Prop}
\label{prop_HBM_intervals}
If in a HBM, $\PPP$ and $\overline P$ are coherent on $\mathcal A(\PP)$, $|\PP|=n$, then $(\PPP,\overline P)$ is the extended probability interval of $\II=(\PPP|_{\PP},\overline P|_{\PP})$.
\end{Prop}

We conclude that NL models are extended probability intervals in very special instances only. 

\section{Conclusions}
\label{sec_conclusions}

In this paper, we have introduced Nearly-Linear models, investigating the consistency properties of their subfamilies and the beliefs they can represent. We have shown that they may range from coherent models (even precise probabilities, in very special instances) to 2-coherent ones, and that they can represent more or less rational ways of assessing uncertainty evaluations. The least rational beliefs elicit ways of thinking of possibly inexperienced or irrational assessors, but are not uncommon in practice. Similar attitudes have been observed also in other areas, like Decision Theory and Behavioural Economics, although models developed in these realms for these purposes only partly overlap with NL models, as we have noticed. It is therefore important to clarify which are the connections between these convictions and the degree of  consistency of the model representing them. In the whole, we find a surprising variety of situations. At one end, the VBM  is a cautious generalisation of the PMM; it is a sound model, as it is coherent and corrects a possible practical shortcoming of the PMM, its neglecting fixed costs in the making of a selling price of an event. At the other end, the HBM and the RRM both convey a subject's conflicting attitudes. These are especially patent with tail events: the $\PPP,\overline P$ evaluations of high (low) $P_0$ probability events are even higher (lower) in the HBM, while they move in the opposite way in the RRM, flattening the $\PPP,\overline P$ values of all non-trivial events in a range narrower than 1. This latter behaviour is less consistent than the former: RRMs are 2-coherent, but not coherent (unless $|\PP|=2$), whilst a HBM may be coherent in some specific circumstances. Both a HBM and a RRM should therefore be regarded as simple ways of explaining certain behaviours, rather than ideal models for assessing uncertainty.

The major open question regarding NL models is how to make inferences, or conditioning, with them. In the case of the VBM,  we deem likely that the results obtained for the PMM in \cite{PVZ} can be generalised.  Extending the remaining 2-coherent models may be more complex, even though some results for the 2-coherent natural extension in \cite{PV16} could be applied.

\section*{Acknowledgements}
We are grateful to Enrique Miranda for fruitful discussions and suggestions and to the referees for some helpful suggestions.

\section*{Appendix: Proofs}

\noindent \textbf{Proof of Lemma \ref{lemma_capacity}.}
	Recalling Definition \ref{def_capac}, 
	only monotonicity has to be checked. Take for this $A,B\in \mathcal A(\PP)$, $A \Rightarrow B$, and let $\mu$ be defined by \eqref{def_mu_NL}. Hence $\mu(\cdot)\in [0,1]$, and, if $\mu(A)=0$, it is obviously $\mu(A)\le \mu(B)$. If $\mu(A)>0$, using \eqref{def_mu_NL}, $b>0$ and monotonicity of $P_0$, we obtain 
	$$
	\mu(A)=\min\{bP_0(A) + a,1\}\le \min\{bP_0(B) + a,1\}=\mu(B).
	$$
\begin{flushright}
	$\blacksquare$
\end{flushright}

\noindent \textbf{Proof of Proposition \ref{conjugate}.}
	$\mu^c(A)$ satisfies Definition \ref{NL_IP} for $A=\emptyset,A=\Omega$. For any $A\in\mathcal A(\PP)\setminus\{\emptyset, \Omega\}$, we have by \eqref{def_mu_NL}
	\begin{equation*}
	\begin{split}
	\mu^c(A) 
	& = 1- \mu(\neg A)=1-\min\{\max\{bP_0(\neg A) + a,0\},1\}\\
	& = 1 + \max\{-\max \{b(1-P_0(A)) + a,0\},-1\}\\
	& = 1 + \max\{\min \{bP_0(A) - (a+b),0\},-1\}\\
	& = \max\{\min \{bP_0(A) - (a+b),0\} + 1,0\}\\
	& = \max\{\min \{bP_0(A) +1 - (a+b),1\},0\}.
	\end{split}
	\end{equation*}
	From the second equality in \eqref{def_mu_NL} and the last expression, we derive that $\mu^c$ is NL$(1-(a+b),b)$, which is the thesis. 
	
\begin{flushright}
	$\blacksquare$
\end{flushright}

\noindent \textbf{Proof of Proposition \ref{prop_2coherence}.}
		Since $\mathcal A(\PP)$ is negation-invariant, we may apply Proposition \ref{prop_character_2coherence}. Taking account of Lemma \ref{lemma_capacity}, there remains to check that,
		\begin{equation}
		\label{equiv_2coherence}
		\forall A\in\mathcal A(\PP), \quad \PPP(A) + \PPP(\neg A)\le 1.
		\end{equation}
		We distinguish the following, exhaustive cases:
		\begin{itemize}
			\item[$(a)$] If $A\in \NNN_{\PPP}$, then $\PPP(A) + \PPP(\neg A)=\PPP(\neg A)\le 1$.
			\item[$(b)$] If $A\in \UUU_{\PPP}$, then $\neg A\in \NNN_{\PPP}$. This is trivially true if $A=\Omega$. Otherwise, $A\in\UUU_{\PPP}$ implies $bP_0(A) + a\ge 1$, equivalent to $b(1-P_0(\neg A)) + a\ge 1$ and to $bP_0(\neg A) +a\le b+2a -1 \le 0$, using \eqref{2coherence} at the last inequality. Hence $\neg A\in \mathcal N_{\PPP}$ and \eqref{equiv_2coherence} holds by $(a)$.
			\item[$(c)$] If $A\in\EEE_{\PPP}$, then $\neg A\notin \UUU_{\PPP}$ (by $(b)$: $\neg A \in \UUU_{\PPP}$ would imply $A\in\NNN_{\PPP}$). If $\neg A\in\NNN_{\PPP}$, we obtain case $(a)$; if $\neg A\in\EEE_{\PPP}$, 
			$$
			\PPP(A)+\PPP(\neg A)=bP_0(A) + a + bP_0(\neg A) + a =b + 2a\le 1.
			$$
		\end{itemize}

\begin{flushright}
	$\blacksquare$
\end{flushright}

\noindent \textbf{Proof of Proposition \ref{equiv_sotto=sopra}.}
	\begin{itemize}
		\item[$(a)$]
		If $b+2a=1$, then $c=1-(a+b)=a$. Hence, $\PPP$ and $\overline P$ are both NL$(a,b)$, so $\overline P=\PPP$, applying Definition \ref{NL_mod}. 
		\item[$(b)$] Consider any event $A\in\mathcal A(\PP)$. Since $\PPP=\overline P=P$, we obtain from \eqref{conju} 
		\begin{equation}
		\label{negazione}
		P(A)+P(\neg A)=1.
		\end{equation}
		It is then not possible that both $A$ and $\neg A$ belong to $\mathcal E_{P}$: otherwise, by \eqref{negazione}, we would get $bP_0(A) + a=P(A)=1-P(\neg A)=1-(bP_0(\neg A) + a)$, equivalent to $b+2a=1$, against the assumption. Hence, at least one between $A$ or $\neg A$ belongs to $\mathcal N_{P}\cup \mathcal U_{P}$, and its probability $P(A)$ is either 0 or 1. But then the probability of its negation is either 1 or 0, respectively, by \eqref{negazione}. 
		Thus $\mathcal E_P=\emptyset$.  
	\end{itemize}
\begin{flushright}
	$\blacksquare$
\end{flushright}

\noindent \textbf{Proof of Proposition \ref{VBM_coherent}.}
	It is sufficient to prove the thesis for $\PPP$. For this, we prove first that $\PPP$ is 2-monotone, i.e., that Equation \eqref{2monot} holds for $\PPP$. Observe for this that:
	\begin{itemize}
		\item Equation \eqref{2monot} holds trivially either when at least one between events $A$ and $B$ is $\Omega$, or when $\PPP(A)\cdot\PPP(B)=0$ (implying by monotonicity of $\PPP$ that $\PPP(A\wedge B)=0$, $\PPP(A\vee B)\ge \max\{\PPP(A),\PPP(B)\}$).
		\item Suppose then $A\neq \Omega,B\neq \Omega,\PPP(A)\cdot\PPP(B)>0$. Equation \eqref{2monot} holds, because 
		\begin{equation*}
		\begin{split}
		\PPP(A) + \PPP(B) 
		& =bP_0(A) + a+ bP_0(B) +a \\
		& =bP_0(A\vee B) + a + bP_0(A\wedge B) +a\\
		& \le \PPP(A\vee B) + \max\{bP_0(A\wedge B) + a,0\}\\
		& =\PPP(A\vee B) + \PPP(A\wedge B).
		\end{split}
		\end{equation*}
	\end{itemize}
	Hence, $\PPP$ is 2-monotone, and $\PPP(\emptyset)=0$, $\PPP(\Omega)=1$. Therefore $\PPP$ is coherent \cite[Corollary 6.16]{TdC}.
\begin{flushright}
	$\blacksquare$
\end{flushright}

\noindent \textbf{Proof of Proposition \ref{Psup_sub_implies}.}
	\emph{Proof of $(a)$.} ($\Rightarrow$) If $A=\emptyset$, the equality is true for $k=0$. Otherwise, $\displaystyle A=\bigvee_{\omega \Rightarrow A, \, \omega \in \PP} \omega$, and, for any such $\omega$ implying $A$, $\overline P(\omega)\le \overline P(A)=0$ by monotonicity of $\overline P$, hence $\overline P(\omega)=0$.
	
	($\Leftarrow$) Let  $\displaystyle A=\bigvee_{i=1}^k \omega_{j_i}$, with $\omega_{j_i}\in\mathcal N_{\overline P}$, $i=1,\dots,k$. Using subadditivity of $\overline P$, it is $\displaystyle \sum_{i=1}^k \overline P(\omega_{j_i})=0\ge \overline P(A)\ge 0$, thus $\overline P(A)=0$, i.e., $A\in\mathcal N_{\overline P}$.
	
	\emph{Proof of $(b)$.} ($\Leftarrow$) Let $\displaystyle A=\omega^+\vee \bigvee_{i=1}^k \omega_{j_i}$, with $\omega^+\in \PP\cap \mathcal E_{\overline P},\,\omega_{j_i}\in \mathcal N_{\overline P}$, 
	$i=1,\dots,k$. Let $k>0$ (otherwise the thesis is trivial). It holds that
	\begin{equation}
	\label{A_omega+}
	1>\overline P(\omega^+)=\overline P(\omega^+) + \sum_{i=1}^k \overline P(\omega_{j_i})\ge \overline P(A)\ge \overline P(\omega^+)>0,
	\end{equation}
	using subadditivity of $\overline P$ at the first weak inequality, its monotonicity at the second. Hence, $\overline P(A)=\overline P(\omega^+)\in\,]0,1[$, meaning that $A\in\mathcal E_{\overline P}$.
	
	($\Rightarrow$) Since $A\in\mathcal E_{\overline P}$, for any $\omega\in\PP$ implying $A$ it is
	$0\le \overline P(\omega)\le \overline P(A)<1$. Thus $\omega\in \mathcal N_{\overline P}\cup\mathcal E_{\overline P}$, and $A$ may be decomposed as
	\begin{equation}
	\label{decomposition}
	A=\bigvee_{\omega\Rightarrow A,\, \omega \in \mathcal N_{\overline P}} \omega \vee \bigvee_{\omega\Rightarrow A,\, \omega \in \mathcal E_{\overline P}}\omega.
	\end{equation}
	Applying subadditivity of $\overline P$,
	\begin{equation*}
	\begin{split}
	\sum_{\omega\Rightarrow A,\,\omega \in \mathcal N_{\overline P}} \overline P(\omega) 
	& + \sum_{\omega\Rightarrow A,\,\omega \in \mathcal E_{\overline P}} \overline P(\omega) \\
	& =\sum_{\omega\Rightarrow A,\,\omega \in \mathcal E_{\overline P}} \Big(bP_0(\omega) + c\Big)\ge \overline P(A)=bP_0(A) + c \\
	& = b\Big(\sum_{\omega\Rightarrow A,\,\omega \in \mathcal N_{\overline P}}P_0(\omega) + \sum_{\omega\Rightarrow A,\,\omega \in \mathcal E_{\overline P}} P_0(\omega)\Big) + c.
	\end{split}
	\end{equation*}
	Name $m$ the cardinality of the set $\{\omega \in \mathcal E_{\overline P}\cap \PP:\omega\Rightarrow A \}$. Then, comparing in the derivation above $\sum_{\omega\Rightarrow A,\,\omega \in \mathcal E_{\overline P}} (bP_0(\omega) + c)$ with the last row term, we obtain
	$$
	mc\ge b\sum_{\omega\Rightarrow A,\,\omega \in \mathcal N_{\overline P}} P_0(\omega) +c \ge c.
	$$
	Since $c<0$, it is either $m=1$, $P_0(\omega)=0$, $\forall \omega \in \mathcal N_{\overline P},$ $\omega\Rightarrow A$, or, alternatively, $m=0$. Assuming $m=0$, by \eqref{decomposition} $A=\bigvee_{\omega\Rightarrow A,\,\omega \in \mathcal N_{\overline P}} \omega$, and by $(a)$ of this proposition $A\in \mathcal N_{\overline P}$, a contradiction.
	Therefore $m=1$, $P_0(\omega)=0$, $\forall \omega \in \mathcal N_{\overline P},$ $\omega\Rightarrow A$, which is the thesis.
	
	\emph{Proof of $(c)$.} Two alternatives may occur:
	\begin{itemize}
		\item[$i)$] $\PP \cap \mathcal U_{\overline P}\neq \emptyset$. Then there is $\omega^*\in\mathcal U_{\overline P}$ such that $bP_0(\omega^*) + c\ge 1$, and \eqref{+*} holds, since $bP_0(\omega^+) + c>0$;
		\item[$ii)$] $\PP \cap \mathcal U_{\overline P} = \emptyset$. 
		Note that, using subadditivity of $\overline P$,
		\begin{multline}
		\label{ge1}
		1 =\overline P(\Omega)=\overline P\Big(\bigvee_{\omega\in\mathcal N_{\overline P}} \omega \vee \bigvee_{\omega\in\mathcal E_{\overline P}} \omega 
		\Big)\\
		\le \sum_{\omega\in\mathcal N_{\overline P}} \overline P(\omega) + \sum_{\omega\in\mathcal E_{\overline P}} \overline P(\omega) 
		= \sum_{\omega\in\mathcal E_{\overline P}} \Big(bP_0(\omega) + c\Big). 
		\end{multline}
		Then, by \eqref{ge1}, since $bP_0(\omega^+)+c<1,$ there exists $\omega^*\in\PP\cap\mathcal E_{\overline P}$, $\omega^*\neq \omega^+$.
		Since $\omega^+\vee \omega^*\notin\mathcal N_{\overline P}$ (by $(a)$) and $\omega^+\vee \omega^*\notin \mathcal E_{\overline P}$ (by $(b)$), necessarily $\omega^+\vee\omega^*\in\mathcal U_{\overline P}$, which implies \eqref{+*}, together with subadditivity:
		$$
		bP_0(\omega^+)+ c + bP_0(\omega^*) + c =\overline P(\omega^+) + \overline P(\omega^*)\ge \overline P(\omega^+\vee \omega^*)= 1.
		$$
	\end{itemize}

\begin{flushright}
	$\blacksquare$
\end{flushright}

\noindent \textbf{Proof of Corollary \ref{unico_e}.}
	Event $A$ can be decomposed as in Proposition \ref{Psup_sub_implies} $(b)$. Assuming that decomposition, it is derived from Equation \eqref{A_omega+} in the proof of Proposition \ref{Psup_sub_implies} $(b)$ that $\overline P(A)=\overline P(\omega^+)$.
\begin{flushright}
	$\blacksquare$
\end{flushright}

\noindent \textbf{Proof of Proposition \ref{cond_sub_HB}.}
	Take any two distinct $\omega_i,\omega_j\in\mathcal E_{\overline P}$. By Proposition \ref{Psup_sub_implies}, 
	$\overline P(\omega_i\vee \omega_j)=1$, implying that 
	$$
	b\big(P_0(\omega_i) + P_0(\omega_j)\big) + c\ge 1.
	$$
	It ensues that ($c<0$):
	$
	P_0(\omega_i) + P_0(\omega_j) \ge \frac{1-c}{b}>\frac{1}{b},
	$
	or also ($P_0(\omega_i)P_0(\omega_j)>0$ from \eqref{upper_HBM})
	\begin{equation}
	\label{bound1_b}
	b>\frac{1}{P_0(\omega_i) + P_0(\omega_j)},
	\end{equation}
	which justifies the internal maximum in Equation \eqref{bound_b}. 
	
	To complete the proof, since $|\PP|\ge 2m$ and assuming for notational simplicity that the first $2m$ elements of $\PP\cap \EEE_{\overline P}$ are $\omega_1,\dots, \omega_{2m}$, 
	we have
	$$
	1=P_0(\Omega)\ge \sum_{i=1}^m \Big(P_0(\omega_{2i-1}) + P_0(\omega_{2i})\Big) >\frac{m}{b},
	$$
	applying the reciprocal of \eqref{bound1_b} at the strict inequality. The inequality $b>m$ follows.
\begin{flushright}
	$\blacksquare$
\end{flushright}

\noindent \textbf{Proof of Proposition \ref{characterise_coherence}.}
We distinguish two cases:
\begin{itemize}
	\item[1)] Let $\PP$ be finite.
	
	If $\overline P$ is coherent, it is necessarily subadditive (cf. \eqref{subadd}). 
	
	Conversely, suppose that $\overline P$ is subadditive. By Theorem \ref{env_thm} $(b)$, $\overline P$ is coherent if we prove that
	$$
	\forall A\in\mathcal A(\PP), \quad \exists P_A:P_A(A)=\overline P(A), \quad P_A(B)\le \overline P(B), \quad \forall B\in\mathcal A(\PP),
	$$
	with $P_A:\mathcal A(\PP)\to\RR$ precise probability. We distinguish three exhaustive alternatives. Prior to this, define  the function
	$$
	\delta(\omega \Rightarrow B)=
	\begin{cases}
	1 & \text{if }\omega \Rightarrow B\\
	0 & \text{if }\omega \not\Rightarrow B
	\end{cases}
	$$
	\begin{itemize}
		\item[$(a)$] $A\in \mathcal E_{\overline P}$.
		
		Define $P_A(\omega)$ on $\PP$ as follows:
		$$
		P_A(\omega)=
		\begin{cases}
		bP_0(\omega^+) + c & \text{if }\omega=\omega^+\\
		1-P_A(\omega^+) & \text{if }\omega=\omega^*\\
		0 & \text{otherwise}
		\end{cases}
		$$
		where $\omega^*,\omega^+$ are the atoms in Proposition \ref{Psup_sub_implies} $(b)$ and $(c)$. Then extend $P_A$ on $\mathcal A(\PP)$ putting, $\forall B\in \mathcal A(\PP),$
		\begin{equation}
		\label{PA}
		P_A(B)=\delta(\omega^+\Rightarrow B)P_A(\omega^+) + \delta (\omega^*\Rightarrow B)P_A(\omega^*).
		\end{equation}
		Clearly, $P_A$ is a probability on $\mathcal A(\PP)$. Let us see that, $\forall B\in \mathcal A(\PP),$ $P_A(B)\le \overline P(B)$:
		\begin{itemize}
			\item[$i)$] if $B\in\mathcal U_{\overline P}$, then $P_A(B)\le 1=\overline P(B)$;
			\item[$ii)$] if $B\in\mathcal N_{\overline P}$, then $\omega^+\not\Rightarrow B$, $\omega^*\not\Rightarrow B$, by Proposition \ref{Psup_sub_implies} $(a)$. From Equation \eqref{PA}, $P_A(B)=0=\overline P(B)$;
			\item[$iii)$] if $B\in\mathcal E_{\overline P}$, by Corollary \ref{unico_e} there is one and only one $\omega_B\in\PP \cap \mathcal E_{\overline P}$ such that $\overline P(B)=\overline P(\omega_B)$.
			
			If $\omega_B=\omega^+$, or  $\omega_B=\omega^*$, then by \eqref{PA} $P_A(B)=P_A(\omega^+)=\overline P(B)$, or $P_A(B)=P_A(\omega^*)=\overline P(B)$, respectively. 
			
			If $\omega_B\neq \omega^+$, $\omega_B\neq \omega^*$, by \eqref{PA} $P_A(B)=0<\overline P(B)$.
		\end{itemize}
		Hence $P_A(B)\le \overline P(B)$ in any case, and, when $B=A$, $P_A(A)=\overline P(A)$: $B=A$ is included into the subcase $\omega_B=\omega^+$ of $iii)$ above.
		
		\item[$(b)$] $A\in \mathcal U_{\overline P}$.
		
		By Proposition \ref{Psup_sub_implies}, two alternatives may occur:
		\begin{itemize}
			\item[$j)$] $\exists \omega^{(1)}\in\mathcal U_{\overline P}: \omega^{(1)}\Rightarrow A$.
			
			Define
			$$
			P_A(B)=\delta\big(\omega^{(1)}\Rightarrow B\big), \quad \forall B\in \mathcal A(\PP).
			$$
			Since $\omega^{(1)}\not\Rightarrow B$ if either $B\in\mathcal N_{\overline P}$ (Proposition \ref{Psup_sub_implies} $(a)$) or $B\in\mathcal E_{\overline P}$ (Proposition \ref{Psup_sub_implies} $(b)$), $P_A(B)=0\le \overline P(B)$ for $B\in \mathcal N_{\overline P} \cup \mathcal E_{\overline P}$. If $B\in \mathcal U_{\overline P}$, $P_A(B)\le 1=\overline P(B)$. Hence it is always $P_A\le \overline P$, and in particular $P_A(A)=1=\overline P(A)$.
			\item[$jj)$] $\exists \omega',\omega''\in\mathcal E_{\overline P}:\omega'\Rightarrow A,\omega''\Rightarrow A, \omega'\neq \omega''$.
			
			Define the following probability:
			\begin{align}
			\nonumber
			P_A(\omega) & =
			\begin{cases}
			bP_0(\omega') + c & \text{if }\omega=\omega'\\
			1-P_A(\omega') & \text{if }\omega=\omega''\\
			0 & \text{if }\omega\notin\{\omega',\omega''\}
			\end{cases}
			& 
			\\
			\label{PA2}
			P_A(B) & =\delta (\omega'\Rightarrow B)P_A(\omega') + \delta (\omega''\Rightarrow B)P_A(\omega''), &  \forall B\in \mathcal A(\PP).
			\end{align}
			To prove that $P_A(B)\le \overline P(B)$, $\forall B\in \mathcal A(\PP)$, note that if $B\in\mathcal U_{\overline P}$ or $B\in\mathcal N_{\overline P}$, the argument parallels that of $i),ii)$, respectively, item $(a)$.
			
			There remains the case $B\in\mathcal E_{\overline P}$, where, by Proposition \ref{Psup_sub_implies} $(b)$, $\displaystyle B=\omega^+ \vee \bigvee_{i=1}^k\omega_{j_i}$, $\omega^+ \in\PP\cap \mathcal E_{\overline P}$, $\omega_{j_i}\in \mathcal N_{\overline P}$, $P_0(\omega_{j_i})=0, \, i=1,\dots, k$. Three subcases may occur (use \eqref{PA2} to compute $P_A$):
			\begin{itemize}
				\item[$\bullet$] if $\omega^+\neq \omega', \,\omega^+\neq \omega'',$ then $P_A(B)=0<\overline P(B)$;
				\item[$\bullet$] if $\omega^+=\omega'$, then $P_A(B)=bP_0(\omega') + c=bP_0(\omega^+) + c=\overline P(B)$,
				with the last equality arising from Corollary \ref{unico_e};
				\item[$\bullet$] If $\omega^+=\omega''$, by Proposition \ref{Psup_sub_implies} $(b)$ $\omega'\vee\omega''\in\mathcal U_{\overline P}$. This fact and subadditivity imply that
				$$
				bP_0(\omega') + c + bP_0(\omega'') + c=\overline P(\omega') + \overline P(\omega'')\ge \overline P(\omega'\vee\omega'')=1.
				$$
				From the inequality above and Corollary \ref{unico_e}, 
				\begin{align*}
				P_A(B) 
				& =P_A(\omega'')=1-P_A(\omega')\\
				& =1-\big(bP_0(\omega') + c\big)\le bP_0(\omega'') + c=\overline P(B).
				\end{align*}
			\end{itemize}
			Finally, $P_A(A)=P_{A}(\omega')+ P_A(\omega'')=1=\overline P(A)$.
		\end{itemize}
		\item[$(c)$] $A\in \mathcal N_{\overline P}$.
		
		For any $A\in\mathcal U_{\overline P}\cup \mathcal E_{\overline P}\supset\{\Omega\}$, the corresponding $P_A$ obtained in $(a)$ and $(b)$ is equal to $0=\overline P(B)$, $\forall B\in\mathcal N_{\overline P}$. Thus, no further $P_A$ has to be added for $A\in \mathcal N_{\overline P}$: $\overline P$ is the upper envelope of the probabilities obtained in the previous two steps.  
\end{itemize}		
\item[2)] Let $\PP$ be infinite. 

To prove coherence of $\overline P$, we have to ensure that for the generic gain $\overline G$ in Definition \ref{def_upper_coherence} it holds that $\max \overline G\ge 0$.
	
	If we consider the \emph{finite} partition $\PP_{G}$ generated by $A_i$, $i=0,1,\dots,n$ 
	(cf. Section \ref{describe_uncertain}),
	we note that the gamble $\overline G$ is defined on $\PP_{G}$, and that, for $i=0,1,\dots,n$, $A_i\in \mathcal A(\PP_{G})$, \emph{which is finite too}. Therefore $\overline G$ is also obtained from checking coherence of the restriction of $\overline P$ on $\mathcal A(\PP_{G})$, $\overline P|_{\mathcal A(\PP_{G})}$. But since $\overline P|_{\mathcal A(\PP_{G})}(A\vee B)\le \overline P|_{\mathcal A(\PP_{G})}(A) + \overline P|_{\mathcal A(\PP_{G})}(B)$, $\forall A,B\in \mathcal A(\PP_{G})$, $\overline P|_{\mathcal A(\PP_{G})}$ is coherent by Part 1), hence $\max \overline G\ge 0$. Since the argument applies to any generic gain concerning $\overline P$, $\overline P$ is coherent too. 
\end{itemize}
\begin{flushright}
	$\blacksquare$
\end{flushright}

\noindent \textbf{Proof of Proposition \ref{Pinf_HB_super}.}
	We prove \eqref{superadd} in all cases but the symmetric ones, obtained exchanging $A$ and $B$. Let then $A,B\in\mathcal A(\PP)$ with $A\wedge B=\emptyset$.
	
	\begin{itemize}
		\item If $A\in\mathcal N_{\PPP}$, \eqref{superadd} boils down to
		$\PPP(A\vee B)\ge \PPP(B),$ true by monotonicity of $\PPP$.
		\item If $A,B\in\mathcal E_{\PPP}$, $\PPP(A) + \PPP(B)=bP_0(A)+ bP_0(B) + 2a=bP_0(A\vee B) + 2a$. 
		Given this,
		\begin{itemize}
			\item[] if $A\vee B\in\mathcal E_{\PPP}$, $\PPP(A) + \PPP(B)=\PPP(A\vee B)+a<\PPP(A\vee B)$ (using \eqref{sign_ab});
			\item[] if $A\vee B\in\mathcal U_{\PPP}$, $\PPP(A) + \PPP(B)\le b+2a\le 1=\PPP(A\vee B)$ (using \eqref{strict_cond_HB}).
		\end{itemize}
		\item It cannot occur that $A\in\mathcal E_{\PPP}\cup \mathcal U_{\PPP}$, $B\in\mathcal U_{\PPP}$: since $A\Rightarrow \neg B$, this would imply 
		$1<\PPP(A) + 1 \le \PPP(\neg B) + \PPP(B)$, which conflicts with Proposition \ref{prop_character_2coherence} $(ii)$.
	\end{itemize}
\begin{flushright}
	$\blacksquare$
\end{flushright}

\noindent \textbf{Proof of Proposition \ref{HBM_coh_implies_alter}.}
	It suffices by conjugacy to prove that $\overline P$ is 2-alternating, checking \eqref{2altern}. Given $E\in\mathcal A(\PP)$, define
	$$
	\mathcal P^E=\{\omega\in\PP:\omega\Rightarrow E,\overline P(\omega)>0\}.
	$$
	Recall from Proposition \ref{Psup_sub_implies} $(b)$ and Corollary \ref{unico_e} that $\overline P(E)=1$ if $|\mathcal P^E|\ge 2$. Now let $\overline P$ be coherent, take any two $A,B\in\mathcal A(\PP)$, and consider $\mathcal P^{A\vee B}$. The following alternatives may occur:
	\begin{itemize}
		\item[$(a)$] $|\mathcal P^{A\vee B}|=0$. Then \eqref{2altern} trivially holds, in the form $0=0$.
		\item[$(b)$] $|\mathcal P^{A\vee B}|\ge 1$ and $|\mathcal P^{A\wedge B}|=0$. Then $\overline P(A\wedge B)=0$, and \eqref{2altern} boils down to $\overline P(A\vee B)\le \overline P(A) + \overline P(B)$, true by subadditivity of $\overline P$. 
		\item[$(c)$] Either $|\mathcal P^{A\vee B}|=|\mathcal P^{A\wedge B}|=1$, or $|\mathcal P^{A\vee B}|\ge 2$ and at least 2 atoms in $\mathcal P^{A\vee B}$ imply $A\wedge B$.
		
		Then all the upper probabilities are identical in \eqref{2altern}, which holds with equality.
		
		\item[$(d)$] $|\mathcal P^{A\vee B}|\ge 2$ and exactly one $\omega \in \mathcal P^{A\vee B}$ implies $A\wedge B$. 
		
		Then $\overline P(A\vee B)=1$, and the following subcases arise:
		\begin{itemize}
			\item[-] All atoms in $\mathcal P^{A\vee B}\setminus\{\omega\}(\neq \emptyset)$ imply $A\wedge \neg B$. Then $\overline P(A)=1$ 
			and \eqref{2altern} reduces to $\overline P(B) - \overline P(A\wedge B)\ge 0$, true by monotonicity of $\overline P$.
			\item[-] All atoms in $\mathcal P^{A\vee B}\setminus\{\omega\}$ imply $\neg A\wedge B$: analogous argument as above.
			\item[-] In $\mathcal P^{A\vee B}\setminus\{\omega\}$, at least one atom implies $A\wedge \neg B$, and at least another one implies $\neg A \wedge B$. Then $\overline P(A)=1=\overline P(B)$, and \eqref{2altern} becomes the true inequality $\overline P(A\wedge B)\le 1$.
		\end{itemize}
	\end{itemize}
\begin{flushright}
	$\blacksquare$
\end{flushright}

\noindent \textbf{Proof of Proposition \ref{prop_N}.}
		Suppose first that $\PPP=\overline P=P$, $P$ probability ($P\neq P_0$). Two alternatives are possible:
		\begin{itemize}
			\item[$(i)$] $\exists \omega^+\in\PP:P(\omega^+)=1$.
			
			Then clearly, since $P$ is a probability, $P(\omega)=0$, $\forall \omega\in\PP\setminus\{\omega^+\}$.
			
			\item[$(ii)$] $\forall \omega\in\PP$, $P(\omega)<1$.
			
			Then $\mathcal E_P\cap \PP\neq\emptyset$ (since $\sum_{\omega\in\PP} P(\omega)=1$).
			
			Take $\omega^*\in\mathcal E_P\cap\PP$: from $\PPP(\omega^*)=bP_0(\omega^*) + a=\overline P(\omega^*)=bP_0(\omega^*) +c$ we immediately obtain $a=c$  ($c<0$, since $P\neq P_0)$.
			
			It is not possible that $|\mathcal E_P\cap \PP|\ge 3$: if, say, $\omega_1,\omega_2,\omega_3\in \mathcal E_P\cap \PP$, then $P(\omega_1\vee\omega_2)=1$, by Proposition \ref{Psup_sub_implies} $(b)$ (thinking of $P$ as a coherent, hence monotone and subadditive, upper probability). This implies $P(\omega_3)=0$, contradicting $\omega_3 \in\mathcal E_P$.
			
			Since $\PP$ is finite, it is also $|\mathcal E_P \cap \PP|\neq 1$. Thus $|\EEE_P\cap \PP|=2$.
		\end{itemize}
		
		Conversely, assume now that either $(a)$ or $(b)$ hold. 
		
		If $(a)$ holds, $\PPP=\overline P=P$ is trivially a probability.
		
		If $(b)$ holds, $a=c$ (that is, $b+2a=1$) ensures that $\PPP=\overline P=P$ (Proposition \ref{equiv_sotto=sopra} $(a)$). $P$ is a probability, by Lemma \ref{zero} (cf. also Remark \ref{rem_zero}).
\begin{flushright}
	$\blacksquare$
\end{flushright}

\noindent \textbf{Proof of Proposition \ref{N+1}.}
	If $\PPP=\overline P=P$ is a probability measure, then since $P$ is coherent as both a lower and an upper probability, $(a)$ and $(b)$ necessarily hold (cf. Equations \eqref{subadd}, \eqref{quasi_superadd}).

	Conversely, let $(a)$, $(b)$ hold. Then, $\forall S\in \mathcal A(\PP)$, it is
	$$
	1=\PPP(\Omega)=\PPP(S\vee\neg S)\le \PPP(S) + \PPP(\neg S)\le 1 + \PPP(S\wedge \neg S)=1,
	$$
	implying that $\PPP(S)=1-\PPP(\neg S)=\overline P(S)$. Thus $\PPP=\overline P=P$. Moreover, $\PPP,\overline P$ are coherent lower, respectively upper, probabilities, from Propositions \ref{characterise_coherence}, \ref{characterise_lower_coherence_gen}. Therefore, $P$ is a probability (Lemma \ref{zero} and Remark \ref{rem_zero}). 

\begin{flushright}
	$\blacksquare$
\end{flushright}

\noindent \textbf{Proof of Proposition \ref{consistency_RRM}.}
	It suffices to consider $\PPP$. $\PPP$ is 2-coherent by \eqref{ab_RRM} and Proposition \ref{prop_2coherence}. 
	
	$\PPP$ is not coherent if $|\PP|\ge 3$. In fact, then there  are distinct $\omega_1,\omega_2\in\PP$, such that $\omega_1\vee \omega_2\neq \Omega$. By \eqref{lower_RRM}, \eqref{ab_RRM}:
	\begin{multline*}
	\PPP(\omega_1)+\PPP(\omega_2)=b\big(P_0(\omega_1) + P_0(\omega_2)\big) +2a \\
	= bP_0(\omega_1\vee \omega_2) + 2a=\PPP(\omega_1\vee\omega_2)+ a >\PPP(\omega_1\vee\omega_2).
	\end{multline*}
	Hence $\PPP$ is incoherent, being not superadditive (cf. \eqref{superadd}).
	
	$\PPP$ is coherent, if $|\PP|=2$. In fact, then $\PP=\{\omega,\neg \omega\}$ and, by \eqref{lower_RRM}, \eqref{ab_RRM}:
	$$
	\PPP(\omega)+\PPP(\neg \omega)=b\big(P_0(\omega) + P_0(\neg \omega)\big) +2a=b+2a\le 1=\PPP(\Omega).
	$$
	From this inequality, $\PPP$ is coherent on $\mathcal A(\PP)$, being the lower envelope of $\{P_1,P_2\}$, with $P_1(\omega)=\PPP(\omega), \, P_1(\neg \omega)=1-\PPP(\omega)$, $P_2(\omega)=1-\PPP(\neg\omega), \, P_2(\neg \omega)\linebreak =\PPP(\neg \omega)$. 
\begin{flushright}
	$\blacksquare$
\end{flushright}

\noindent \textbf{Proof of Proposition \ref{coherence_implies_2monotonicity}.}
	Follows from: Proposition \ref{consistency_RRM}, recalling that any coherent $\PPP$ ($\overline P$) on $\mathcal A(\PP)$ is 2-monotone (2-alternating) if $|\PP|\le 3$ \cite[Proposition 6.9]{TdC}, Proposition \ref{VBM_coherent}, Proposition \ref{HBM_coh_implies_alter}.
\begin{flushright}
	$\blacksquare$
\end{flushright}

\noindent \textbf{Proof of Proposition \ref{prop_deg_mod}.}
	\begin{itemize}
		\item[$(a)$] It is simple to verify that $\PPP_h$ obeys the conditions $(i),(ii),(iii)$ in Proposition \ref{prop_character_2coherence}, hence it is 2-coherent.
		\item[$(b)$] We prove that Definition \ref{def_all} $(b)$ applies to $\PPP_h|_{\mathcal A(\PP)\setminus\{\emptyset, \Omega\}}$. Considering a generic $\underline G_{C}$, recalling that $\sum_{i=1}^ns_i=1$ there, 
		$$
		\underline G_{C}=\sum_{i=1}^n s_i\big(I_{A_i} - a\big) - \big(I_{A_0} - a\big)= \sum_{i=1}^n s_i \, I_{A_i} - I_{A_0},
		$$
		there is $\omega^*\in\PP$ such that $\omega^*\wedge A_0=\emptyset$ (since $A_0$ cannot be $\Omega$). It follows that 
		$$
		\max \underline G_C\ge \underline G_C(\omega^*)=\sum_{i=1}^n s_i \, I_{A_i}(\omega^*) - I_{A_0}(\omega^*)=\sum_{i=1}^n s_i \,  I_{A_i}(\omega^*)\ge 0.
		$$
		\item[$(c)$] We prove $(c)$ by means of the following chain of implications: if $\PP$ is finite, 
		$$
		\PPP_h \text{ is C-convex} \Rightarrow \PPP_h \text{ avoids sure loss} \Rightarrow a\le \tfrac{1}{n} \Rightarrow \PPP_h \text{ is C-convex.}
		$$
		In fact, the first implication is a property of C-convex probabilities \cite{PV03}.
		
		To prove the second implication, let $\PPP_h$ avoid sure loss. Applying Definition \ref{def_all} $(c)$, it must hold that $$
		\max \underline G_{\rm ASL}=\max \sum_{i=1}^n \frac{1}{n}(I_{\omega_i} - a)\ge 0.
		$$
		Since, $\forall \omega_i\in\PP$, $\underline G_{\rm ASL}(\omega_i)=\frac{1}{n}(1-a)+ \frac{n-1}{n}(-a)=\frac{1}{n} - a$, it is $\max \underline G_{\rm ASL}\ge 0$ iff $a\le \frac{1}{n}$.
		
		For the third implication, let  $a\le \frac{1}{n}$. We prove that $\PPP$ is convex applying Theorem \ref{env_thm} $(c)$, with $\mathcal M=\{P_0,P_1,\dots, P_n\}$ and $\alpha:\mathcal M\to\RR$ defined as follows: $\forall i,j=1,\dots,n,$
		\begin{align*}
		P_0(\omega_i) & = \frac{1}{n}, 
		& \alpha_0=\alpha(P_0)=0,\\
		P_i(\omega_j) & =
		\begin{cases}
		1 & \text{ if }\omega_j=\omega_i\\
		0 & \text{ if } \omega_j\neq \omega_i
		\end{cases},
		& \alpha_i=\alpha(P_i)=a.
		\end{align*}
		In fact, it holds that 
		\begin{equation}
		\label{appl_env}
		\PPP_h(A)=\min\{P_0(A),P_1(A)+\alpha_1,\dots,P_n(A) + \alpha_n\}, \quad \forall A\in\mathcal A(\PP).
		\end{equation}
		To see this, note that, for $i=1,\dots,n$ and for all $A\in\mathcal A(\PP)\setminus\{\emptyset, \Omega\}$, we have
		\begin{align*}
		P_i(A) +\alpha_i & = 1+a >a=\PPP_h(A) & \text{if }\omega_i\Rightarrow A,\\
		P_i(A) +\alpha_i & = a =\PPP_h(A) & \text{if }\omega_i\Rightarrow \neg A,
		\end{align*}
		while for $A=\emptyset$, $i=1,\dots,n$,
		$$
		P_i(\emptyset)+\alpha_i=a>0, \quad P_0(\emptyset)+\alpha_0=0=\PPP_h(\emptyset),
		$$
		and for $A=\Omega$, $i=1,\dots,n$,
		$$
		P_i(\Omega)+\alpha_i=1+a>1, \quad P_0(\Omega)+\alpha_0=1=\PPP_h(\Omega).
		$$
		Lastly, we have $\forall A\in\mathcal A(\PP)\setminus\{\emptyset, \Omega\}$, 
		$$
		P_0(A)\ge P_0(\bar \omega)=\frac{1}{n}\ge a=\PPP_h(A),
		$$
		where $\bar\omega\in\PP,\,\bar\omega\Rightarrow A$.
		
		Thus, any of $P_0,P_i+\alpha_i$, $i=1,\dots,n$, is not smaller than $\PPP_h$ at any event $A$, with equality achieved by $P_0$ if $A=\emptyset,\, A=\Omega$, by a convenient $P_i+\alpha_i$ (such that $\omega_i\Rightarrow \neg A$) otherwise. This means that \eqref{appl_env} holds, and that $\PPP_h$ is convex. Since $\PPP_h(\emptyset)=0$, it is also C-convex. 
	\end{itemize}
\begin{flushright}
	$\blacksquare$
\end{flushright}

\noindent \textbf{Proof of Proposition \ref{ideals_and_others}.}
	\begin{itemize}
		\item[$(a)$] The proof that $\UUU_{\PPP}$ is a filter is given in \cite[Section 2.9.8]{W} with reference to a 0-1 valued lower probability, but applies without modifications to a generic coherent $\PPP$. $\NNN_{\PPP}$ is generally not an ideal, since $\PPP(A)=\PPP(B)=0$ does not imply $\PPP(A\vee B)=0$.
		
		Equation \eqref{U_implies_N} holds, as a consequence of $\PPP(A)+\PPP(\neg A)\le 1$ and of non-negativity of $\PPP$, both necessary conditions for 2-coherence, hence for coherence.
		
		\item[$(b)$] It may be checked, by inspecting all admissible gains in Definition \ref{def_all} $(d)$, that the assignment $\PPP(A)=\PPP(B)=1$, $\PPP(A\wedge B)=1-\varepsilon$, is 2-coherent on $\mathcal D=\{A,B,A\wedge B\}\subset \mathcal A(\PP)$, $\forall \varepsilon\in [0,1]$. Whatever $\varepsilon$ is chosen, $\PPP$ admits a 2-coherent extension on $\mathcal A(\PP)$ \cite{PV16}. In particular this holds for $\varepsilon=1$, i.e., $\PPP(A\wedge B)=0$. Equation \eqref{U_implies_N} holds for the reasons elicited in $(a)$.
		
		\item[$(c)$] To show that $\PPP(A)=\PPP(B)=1$ implies $\PPP(A\wedge B)\ge \frac{1}{2}$ ($A\wedge B\neq \emptyset$), take $\PPP:\mathcal D\to\RR$ as in $(b)$ (in particular, $\PPP(A\wedge B)=1-\varepsilon$, $\varepsilon\ge 0$) and the related following gain $\underline G$, admissible by Definition \ref{def_all} $(b)$ (where $s_1=s_2=\frac{1}{2}$):
		$$
		\underline G=\tfrac{1}{2}\big(I_{A}-1\big)+ \tfrac{1}{2} \big(I_{B}-1\big)-\big(I_{A\wedge B}-(1-\varepsilon)\big)=\tfrac{1}{2}\big(I_A + I_B\big)-I_{A\wedge B}-\varepsilon.
		$$
		Since $\underline G(A\wedge B)=\underline G(\neg A\wedge \neg B)=-\varepsilon\le 0$, $\underline G(A\wedge \neg B)=\underline G(\neg A\wedge B)=\frac{1}{2}-\varepsilon$, when $\varepsilon>0$ it is $\max \underline G\ge 0$ iff $\varepsilon\le \frac{1}{2}$. 
		Thus, necessarily, $\PPP(A\wedge B)\ge \frac{1}{2}$.
		
		The bound $\PPP(A\wedge B)=\frac{1}{2}$ may be achieved, for instance by $\PPP$ in Table \ref{convex_P}. $\PPP$ is convex because it is the lower envelope of $P_1+\alpha_1$, $P_2+\alpha_2$, with $\alpha_1=\alpha(P_1)=0$, $\alpha_2=\alpha(P_2)=\frac{1}{2}$ (Theorem \ref{env_thm} $(c)$).
		\begin{table}[htbp!]
			\label{ex_precise}
			\begin{center}
				\begin{tabular}{c|c|c|c|c|c|c}
					&&&&&&\\[-1em] & $A\wedge B$ & $\neg A\wedge B$ & $A\wedge \neg B$ & $\neg A\wedge \neg B$ & $A$ & $B$ \\
					&&&&&&\\[-1em] \hline &&&&&&\\[-1em]
					$P_1=P_1+\alpha_1$ & 1 & 0 & 0 & 0 & 1 & 1 \\
					&&&&&&\\[-1em] \hline &&&&&&\\[-1em]
					$P_2$ & 0 & $\frac{1}{2}$ & $\frac{1}{2}$ & $0$ & $\frac{1}{2}$ & $\frac{1}{2}$ \\
					&&&&&&\\[-1em] \hline &&&&&&\\[-1em]
					$P_2+\alpha_2$ & $\frac{1}{2}$ & $1$ & $1$ & $\frac{1}{2}$ & $1$ & $1$ \\
					&&&&&&\\[-1em] \hline &&&&&&\\[-1em]
					\hline &&&&&&\\[-1em]
					$\PPP$ & $\frac{1}{2}$ & $0$ & $0$ & $0$ & $1$ & $1$
				\end{tabular}
				\caption{A convex lower probability with $\PPP(A)=\PPP(B)=1,$ $\PPP(A\wedge B)=\frac{1}{2}$.}
				\label{convex_P}
			\end{center}
		\end{table}
		
		Finally, \eqref{U_implies_N} may not hold. It does not, for instance, for $\PPP'=P_2+\alpha_2$ (which is convex, as a lower envelope of itself) in Table \ref{convex_P}: $\PPP'(A\wedge \neg B)=1$, but $\PPP'(\neg (A\wedge \neg B))
		=1$. 
		
		Equation \eqref{U_implies_N} anyway applies, if $\PPP$ is C-convex. In fact, $\PPP$ then also avoids sure loss \cite[Proposition 3.5 $(e)$]{PV03}, and as such satisfies  $\PPP(A)+\PPP(\neg A)\le 1$, while $\PPP$ is non-negative by C-convexity.
	\end{itemize}

\begin{flushright}
	$\blacksquare$
\end{flushright}

\noindent \textbf{Proof of Proposition \ref{characterisation_extended}.}
	Obviously, $\II$ is coherent, being the restriction of the coherent $(\PPP,\overline P)$ on $\PP$. Therefore, its extended probability interval $\II_E$ is given by \eqref{extended_lower}, \eqref{extended_upper}, with
	$$
	l_i=\PPP(\omega_i), \quad u_i=\overline P(\omega_i), \quad i=1,\dots,n.
	$$
	Note that by conjugacy (of both $\PPP, \overline P$ and $l,u$) it is enough to establish that
	\begin{equation}
	\label{Pl}
	\PPP(A)=l(A), \quad \forall A\in \mathcal A(\PP)
	\end{equation}
	iff $(a), (b)$ or $(c)$ hold. 
	Equality \eqref{Pl} is certainly true if $A=\Omega$, by coherence of $\PPP,l$: $\PPP(\Omega)=l(\Omega)=1$, and, if $\PPP(A)=0$, using also \eqref{N}: $\PPP(A)=l(A)=0$.
	
	Therefore, it remains to check when 
	$$
	\PPP(A)=\max\bigg\{\sum_{\omega_i\Rightarrow A} l_i, 1- \sum_{\omega_i\Rightarrow \neg A} u_i\bigg\}
	$$
	for those events $A$ such that $\PPP(A)>0$, $A\neq \Omega$.
	
	Since by \eqref{N2} $l(A)\le \PPP(A)$, this is equivalent to check when $\displaystyle \PPP(A)=\sum_{\omega_i\Rightarrow A} l_i$ or $\displaystyle \PPP(A)=1- \sum_{\omega_i\Rightarrow \neg A} u_i$.
	Taking then one such $A$, we investigate first when $\displaystyle \PPP(A)=\sum_{\omega_i\Rightarrow A} l_i$. 
	
	Defining 
	\begin{align*}
	\nonumber
	\mathcal P_A^+ & =\{\omega_i\in\PP:\omega_i\Rightarrow A,\PPP(\omega_i)>0\}=\big\{\omega_i\in\PP:\omega_i\Rightarrow A,P_0(\omega_i)>-\tfrac{a}{b}\big\},\\
	\mathcal P_A^0 & =\{\omega_i\in\PP:\omega_i\Rightarrow A,\PPP(\omega_i)=0\}=\big\{\omega_i\in\PP:\omega_i\Rightarrow A,P_0(\omega_i)\le -\tfrac{a}{b}\big\},
	\end{align*}
	observe that, by \eqref{lower_VBM} and letting $|\mathcal P_A^+|=m$,
	\begin{equation}
	\begin{split}
	\label{let}
	\sum_{\omega_i\Rightarrow A}l_i 
	& = 
	\sum_{\omega_i\in \mathcal P_A^0} \PPP(\omega_i) + \sum_{\omega_i\in \mathcal P_A^+}\PPP(\omega_i)  \\
	& =\sum_{\omega_i\in\mathcal P_A^+} \Big(bP_0(\omega_i)+a\Big) 
	= ma + b \sum_{\omega_i\in\mathcal P_A^+} P_0(\omega_i).
	\end{split}
	\end{equation}
	Thus, recalling again \eqref{lower_VBM}, $\displaystyle\PPP(A)=\sum_{\omega_i\Rightarrow A} l_i$ if and only if
	\begin{equation}
	\label{Pl2}
	\PPP(A)=bP_0(A)+a= b\sum_{\omega_i\Rightarrow A}P_0(\omega_i) + a = ma + b \sum_{\omega_i\in\mathcal P_A^+} P_0(\omega_i)=\sum_{\omega_i\Rightarrow A} l_i.
	\end{equation}
	Note that \eqref{Pl2} is false when $m=0$: in fact, it is $\PPP(A)>0=\sum_{\omega_i\Rightarrow A} l_i$.
	
	Thus, the equality \eqref{Pl2} holds iff one of the following two cases applies:
	\begin{itemize}
		\item $a=0$ (case $(a)$).
		
		In fact, then $\PPP(\cdot)=b P_0(\cdot)$, which implies that $P_0(\omega_i)>0$ iff $\PPP(\omega_i)>0$. Using this fact at the second next equality and \eqref{let} at the third,
		$$
		\PPP(A)=b\sum_{\omega_i\Rightarrow  A, P_0(\omega_i)>0} P_0(\omega_i)=b\sum_{\omega_i\in\mathcal P_A^+} P_0(\omega_i)=\sum_{\omega_i \Rightarrow A} l_i.
		$$
		\item $a<0$, $\displaystyle A=\omega^+\vee \bigvee_{j=1}^k \omega_{i_j}$, $k\in \{0,\dots,n-1\}$, $\PPP(\omega^+)>0$, $P_0(\omega_{i_j})=0$, $j=1,\dots,k$ (case $(c2)$).
		
		To see this, note that it is $\displaystyle\sum_{\omega_i\in\mathcal P_A^+} P_0(\omega_i)\le \sum_{\omega_i\Rightarrow A} P_0(\omega_i)$, and that $ma<a$ for $m\ge 2$, since $a<0$. Thus, for $a<0$ we need to require $\displaystyle m=1, \, \sum_{\omega_i\in\mathcal P_A^+} P_0(\omega_i) = \sum_{\omega_i\Rightarrow A} P_0(\omega_i)$ for \eqref{Pl2} to hold. These conditions can be equivalently restated in the form of case $(c2)$.
	\end{itemize}
	To establish now when $\displaystyle \PPP(A)=1-\sum_{\omega_i\Rightarrow \neg A}u_i$, note first that $\PPP(A)>0$ iff $\overline P(\neg A)<1$, so that, since here $\PPP(A)>0$, we have by \eqref{upper_VBM} that
	\begin{equation}
	\label{overP}
	\overline P(\omega_i)=bP_0(\omega_i) + c<1, \quad \forall \omega_i\in\PP: \omega_i\Rightarrow \neg A. 
	\end{equation}
	Using \eqref{overP}, we may write
	\begin{equation}
	\label{1meno}
	1-\sum_{\omega_i \Rightarrow \neg A} u_i=1-\sum_{\omega_i \Rightarrow \neg A} \big(bP_0(\omega_i) + c\big) = 1- rc-b\sum_{\omega_i \Rightarrow \neg A}P_0(\omega_i),
	\end{equation}
	where $r=|\{\omega_i\in\PP: \omega_i\Rightarrow \neg A\}|$.
	
	On the other hand, still using \eqref{upper_VBM}, we obtain
	\begin{equation}
	\label{overP2}
	\PPP(A)=1-\overline P(\neg A)=1-c-bP_0(\neg A)=1-c-b\sum_{\omega_i\Rightarrow \neg A}P_0(\omega_i).
	\end{equation} 
	Comparing \eqref{1meno} and \eqref{overP2} (recall that $c\ge 0$), it is now easy to see that $\displaystyle\PPP(A)=1-\sum_{\omega_i \Rightarrow \neg A} u_i$ if and only if one of the following applies:
	\begin{itemize}
		\item $c=0$ (which is case $(b)$, $a+b=1$),
		\item $r=1$, equivalent to $\neg A\in \PP$ (case $(c3)$).
	\end{itemize}

\begin{flushright}
	$\blacksquare$
\end{flushright}

\noindent \textbf{Proof of Proposition \ref{prop_HBM_intervals}.}
	$\II$ is coherent for the same reasons as at the beginning of the proof of Proposition \ref{characterisation_extended}, and (again by conjugacy) it suffices to prove that
	$$
	\overline P(A)=u(A), \quad \forall A\in\mathcal A(\PP),
	$$
	with $u(A)$ given by \eqref{extended_upper}, $u_i=\overline P(\omega_i)$, $l_i=\PPP(\omega_i)$, $i=1,\dots, n$.
	
	The following exhaustive alternatives are to be considered:
	\begin{itemize}
		\item $\overline P(A)=1$. 
		
		From \eqref{N2} and coherence of $\overline P$, $u$, we get $\overline P(A)=u(A)=1$.
		
		\item $\overline P(A)=0$.
		
		If $A=\emptyset$, it is $\overline P(\emptyset)=u(\emptyset)=0$, by coherence.
		
		If $A\neq \emptyset$, also $\overline P(\omega_i)=0$, $\forall \omega_i\in\PP, \, \omega_i\Rightarrow A$. Therefore $\displaystyle\sum_{\omega_i\Rightarrow A}u_i=0$ and
		$$
		0\le u(A)=\min\bigg\{0,1-\sum_{\omega_i\Rightarrow \neg A} l_i\bigg\}=0=\overline P(A),
		$$
		thus $\overline P(A)=u(A)=0$.
		
		\item $0<\overline P(A)<1$.
		
		Exploiting \eqref{N2} at the first inequality, Corollary \ref{unico_e} at the first equality ($\omega^+$ is such that $\overline P(\omega^+)>0$, $\omega^+\Rightarrow A$) and Proposition \ref{Psup_sub_implies} $(b)$ at the second, we obtain:
		\begin{equation*}
		\begin{split}
		u(A) 
		& \ge \overline P(A)=\overline P(\omega^+)\\
		& =\sum_{\omega_i\Rightarrow A}\overline P(\omega_i)\ge \min\bigg\{\sum_{\omega_i\Rightarrow A}\overline P(\omega_i),1-\sum_{\omega_i\Rightarrow \neg A}l_i\bigg\}=u(A),
		\end{split}
		\end{equation*}
		which implies $\overline P(A)=u(A)$.
	\end{itemize}

\begin{flushright}
	$\blacksquare$
\end{flushright}







\bibliographystyle{apa}
\bibliography{nlum_CPV}{}






\end{document}